\documentclass[12pt]{article}
\usepackage{amsmath,enumerate,amsfonts,color,amssymb,amsthm,mathtools,mdframed,esint,dsfont}
\usepackage[toc,page]{appendix}

\usepackage{tikz}

\def\NN{{\mathbb N}}

\def\RR{{\mathbb R}}

\def\mcA{{\mycal A}}

\def\eps{{\varepsilon}}

\newtheorem{theorem} {\sc  Theorem\rm} [section]

\newtheorem{lemma} [theorem] {\sc  Lemma\rm}
\newtheorem{proposition} [theorem] {\sc  Proposition\rm}

\newcounter{marnote}

\DeclareFontFamily{OT1}{rsfs}{}
\DeclareFontShape{OT1}{rsfs}{m}{n}{ <-7> rsfs5 <7-10> rsfs7 <10-> rsfs10}{}
\DeclareMathAlphabet{\mycal}{OT1}{rsfs}{m}{n}

\def\bS{\mathbb{S}}

\def\be{\begin{equation}}
\def\ee{\end{equation}}

\def\mcE{{\mycal E}}

\newcommand{\R}{\mathbb{R}}

\newcommand{\N}{\mathbb{N}}

\newcommand{\SSS}{\mathbb{S}}

\newcommand{\mbbE}{\mathbb{E}}
\newcommand{\mcU}{\mathcal{U}}\newcommand{\mcH}{\mathcal{H}}

\newcommand{\lip}[1]{{\rm Lip}(#1)}
\newcommand{\liptwo}[1]{\left\|#1 \right\|_{C^2}}
\newcommand{\dbar}[1]{\overline{\overline{#1}}}

\DeclarePairedDelimiter\abs{\lvert}{\rvert} 
\DeclarePairedDelimiter\norm{\lVert}{\rVert}
\makeatletter
\let\oldabs\abs
\def\abs{\@ifstar{\oldabs}{\oldabs*}}
\let\oldnorm\norm
\def\norm{\@ifstar{\oldnorm}{\oldnorm*}}
\makeatother
\def\be{\begin{equation}}
\def\ee{\end{equation}}
\def\bea#1\eea{\begin{align}#1\end{align}}
\def\non{\nonumber}

\newcommand{\BB}{\mathbb{B}}
\newcommand{\tK}{{\tilde K}}

\newmdtheoremenv{boxtheorem}[theorem]{Theorem}
\newmdtheoremenv{boxlemma}[theorem]{Lemma}

\numberwithin{equation}{section}

\begin{document}
\title{Sphere-valued harmonic maps with  surface energy and the $K_{13}$ problem}

\author{Stuart Day\thanks{ Department of Mathematics, University of Sussex, Pevensey III, Falmer, BN1 9QH, UK \qquad\text{Email: S.Day@sussex.ac.uk}} and Arghir Dani Zarnescu \thanks{ Ikerbasque, Basque Foundation for Science, Maria Diaz de Haro 3,
48013, Bilbao, Bizkaia, Spain}\,\,\thanks{BCAM, Basque Center for Applied Mathematics, Mazarredo 14, 48009 Bilbao, Bizkaia, Spain\qquad\text{Email: azarnescu@bcamath.org}}\,\,\thanks{ ``Simion Stoilow" Institute of Mathematics of the Romanian Academy, 21 Calea Grivi\c{t}ei Street,  010702 Bucharest, Romania}}



\maketitle


\begin{abstract}
	We consider an energy functional  motivated by the celebrated {\it $K_{13}$ problem} in  the Oseen-Frank theory of nematic liquid crystals.	It is defined for sphere-valued functions and 
	appears as the usual Dirichlet energy with an additional surface term.  
	\par  It is known that this energy is unbounded from below and our aim has been to study the local minimizers. We show that even having  a critical point in a suitable energy space imposes severe restrictions on the boundary conditions. Having suitable boundary conditions makes
	 the energy functional bounded and in this case we study the partial regularity of the minimizers. \end{abstract}

\section{Introduction}
In this paper we study critical points of the following energy functional 

\be\label{def:Energ}
\mbbE[n]=\int_\Omega \frac{K}{2}|\nabla n|^2\,dx+\tK_{13}\int_{\partial\Omega}((n\cdot \nabla)n)\cdot\nu\,d\sigma 
\ee for functions $n$ with $|n(x)|=1$ a.e. in $\bar\Omega$ where 

$$K>0, \tK_{13}\in \R$$ and $\Omega\subset\R^d$ with $d\in\{2,3\}$ is a $C^2$ domain and $\nu$ denotes the exterior unit-normal.

This functional  is motivated  by the {\it  $K_{13}$ problem} in the Oseen-Frank theory of liquid crystals. More details about the physical  relevance of this problem are provided in the next subsection, Section~\ref{subsec:physmot}. 

    It has been known since $1985$  thanks the work of C. Oldano and G. Barbero \cite{OldanoBarbero85} that there exists a domain $\Omega\subset\R^3$ and a sequence $\{n_k\}_{k\in\N}\subset C^\infty(\Omega;\bS^2)$ such that $\mbbE[n_k]\to -\infty$ as $k\to\infty$. Thus  one cannot understand the physical meaning of the energy in the standard way, i.e. from the point of view of global energy minimizers. However it is conceivable that the energy might still have  nontrivial local energy-minimzers and this has been
the starting point  of this work.

A first  question is then to understand what is the space in which to look for local minimizers. In order to understand this one can start by noting that for $n\in C^2(\Omega,\bS^{d-1})$ the energy becomes:

                    $$\mbbE[n]=\int_\Omega \frac{K}{2}|\nabla n|^2\,dx+\tK_{13}\int_\Omega \sum_{\alpha,\beta=1}^d \left(\frac{\partial n_\alpha}{\partial x_\beta}\frac{\partial n_\beta}{\partial x_\alpha}+n_\beta\frac {\partial^2 n_\alpha}{\partial x_\alpha \partial x_\beta}\right)\,dx$$

A minimal requirement for the functional space is that the energy makes sense for functions in it. Thus a natural choice is:
\be\label{def:mcA}
\mcA:=W^{1,2}(\Omega;\bS^{d-1})\cap W^{2,1}(\Omega;\bS^{d-1})
\ee

In this space one can consider various boundary conditions, which would make the space smaller. However we surprisingly have that there are severe constraints on what the boundary conditions could be:

\medskip
\begin{theorem}\label{thm:bdrydet}

Let $\Omega \subset \RR^d$, $d=2,3$ be a $C^1$ set with unit-norm exterior normal denoted $\nu$. Consider the energy $\mbbE$ as defined in \eqref{def:Energ}.

 Let $\bar n$ be a critical point of  $\mbbE$ in the functional space $\mcA$ defined in \eqref{def:mcA}. Then we have

\be
\bar n(x)\cdot\nu(x)\in \{0,\pm 1\}, \textrm{ for almost all }x\in\partial\Omega
\ee

\end{theorem}

The most interesting case is the one  when $\bar n(x)\cdot \nu(x)=0$ for almost all $x\in\partial\Omega$ and in dimension $d=3$ as this allows for a certain level of  freedom at the boundary. Then one can show that the energy $\mbbE$ reduces to:

\be\label{def:energG}
\mathbb{G}[n]:= \int_\Omega K\sum_{\alpha,\beta=1}^3 \frac{\partial n_\alpha}{\partial x_\beta}\frac{\partial n_\alpha}{\partial x_\beta} - \tK_{13}\int_{\partial \Omega} \sum_{\alpha,\beta=1}^3 \frac{ \partial \nu_\beta}{\partial x_\alpha} n_\beta n_\alpha.
\ee
This energy makes sense in the large functional space $W^{1,2}(\Omega,\SSS^2)$ and it is easily shown to be bounded from below and lower semicontinuous. Thus one  can now consider minimizing $\mathbb{G}$ over the function space of $W^{1,2}$ functions with tangential boundary conditions: 
\be\label{def:mcU}
\mcU :=\{u \in W^{1,2}(\Omega, \SSS^2): {\rm Trace}(u) \in \mathcal{T} \} 
\ee
where 
\begin{equation}\label{def:mcT}
\mathcal{T}:= \{\gamma \in H^{\frac{1}{2}}(\partial \Omega , \SSS^2): \gamma(x)\cdot \nu(x) = 0 \text{ for almost every } x \in \partial \Omega\}.
\end{equation}

A first issue to consider is weather or not this functional space is non-empty as the topological constraint might make it empty, as shown by the {\it ``Hairy Ball" Theorem}. Fortunately in our case the regularity at the boundary is weaker than continuous  and we have:

\begin{proposition}\label{prop:nonemptyfspace}
Let $\Omega$ be a bounded $C^2$ domain in $\R^3$. Then the space $\mcU$ defined in \eqref{def:mcU} is non-empty.
\end{proposition}

Standard arguments provide the existence of a global minimizer. In general this minimizer might not be continuous at the boundary for topological reasons (think of the example of the ``Hairy Ball" theorem). It is then of interest to look into the matter of partial regularity for the global minimizers of $\mathbb{G}$ in the space $\mcU$ .  This is related to the works of  R. Hardt and F. Lin in \cite{HardtLin89} and later  that of \cite{Scheven06} who considered partially constrained boundary conditions, though only for the Dirichlet functionals. We look into this through a method combining the two approaches in the works mentioned above and taking into account the effect
of the surface energy. We can thus show:

\begin{theorem}\label{thm:preg}
Let $\Omega\subset\R^3$ be a $C^2$ domain. Then a global energy minimizer of the energy $\mathbb{G}$ (defined  in \eqref{def:energG}) in the space \eqref{def:mcU} is continuous on $\bar\Omega\setminus Z$ where $Z$ is a set of one-dimensional Hausdorff measure equal to zero.
 \end{theorem}

The paper is organized as follows: in the next section  physical background is provided, to be followed in Section~\ref{sect:constraints} by the example of Barbero and Oldano showing the unboundedness of the energy functional, and then the proof of Theorem~\ref{thm:bdrydet}. In the last part, Section~\ref{sect:preg}, we prove Proposition~\ref{prop:nonemptyfspace} and then  Theorem~\ref{thm:preg}. The appendices contain a number of technical lemmas and the list of notations.

\subsection{Physical motivation}
\label{subsec:physmot}

Nematic liquid crystals are the simplest yet the most used type of liquid crystals, with wide-ranging applications, particularly in displays. The simplest and most comprehensive model used for describing the stationary patterns is related to the Oseen-Frank energy (see for instance \cite{Virgabook}) :

\bea\label{OFenerg}
\mbbE_{OF} [n]=\int_\Omega K_1|\nabla\cdot n|^2&+K_2|n\cdot(\nabla\times n)|^2+K_3|n\times(\nabla\times n)|^2\,dx\\
&+\int_\Omega (K_2+K_{24})(\textrm{tr}(\nabla n)^2-(\nabla\cdot n)^2)+K_{13}\nabla\cdot \left((\nabla\cdot n)n  \right)\,dx
\eea where the vector $n$ is unit-length. 
Using the identity, valid for $n\in C^1(\Omega;\mathbb{S}^2)$:

$$\textrm{tr}(\nabla n)^2+|n\cdot(\nabla\times n)|^2+|n\times (\nabla\times n)|^2=|\nabla n|^2
$$ we have that for ``equal elastic constants" $K_1=K_2=K_3$ the Oseen-Frank energy reduces to 

$$\tilde\mbbE_{OF}[n]=\int_\Omega K_2|\nabla n|^2+K_{24}(\textrm{tr}(\nabla n)^2-(\nabla\cdot n)^2)+K_{13}\nabla\cdot \left((\nabla\cdot n)n  \right)\,dx
$$

The $K_{24}$ term is a null-Lagrangian as we have (see for instance \cite{Virgabook},\cite{alouges1997minimizing}):

\bea
\int_\Omega K_{24}(\textrm{tr}(\nabla n)^2-(\nabla\cdot n)^2)\,dx&=\int_\Omega \nabla\cdot\left((n\cdot\nabla)n -(\nabla\cdot n)n\right)\,dx\non\\
=\int_{\partial\Omega} \left[ (n\cdot\nabla) n-(\nabla\cdot n) n\right]\cdot\nu\,d\sigma&=\int_{\partial\Omega} D_t n: n\otimes \nu- \textrm{tr}(D_t n) n\cdot \nu\,d\sigma\non
\eea where $D_t n:=\nabla n-\nu\otimes \frac{\partial n}{\partial \nu}$ is a differential operator that involves only tangential derivatives, thus its value depends only on the boundary conditions.

The $K_{13}$ term is different as it can be expressed as a surface integral, but the surface integral {\it does not depend only on tangential derivatives}. However, we can remove the tangential contribution to the $K_{13}$ term as follows:

\bea
 K_{24}(\textrm{tr}(\nabla n)^2-(\nabla\cdot n)^2)+K_{13}\nabla\cdot \left((\nabla\cdot n)n  \right)&= (K_{24}-K_{13})(\textrm{tr}(\nabla n)^2-(\nabla\cdot n)^2)\non\\
&+ K_{13}(\textrm{tr}(\nabla n)^2-(\nabla\cdot n)^2)+K_{13}\nabla\cdot \left((\nabla\cdot n)n  \right)\,\non\\
&=\tilde K_{24}(\textrm{tr}(\nabla n)^2-(\nabla\cdot n)^2)+K_{13}\nabla\cdot\left((n\cdot\nabla) n\right)\,\non\\
\eea thus the energy $\mbbE$ that we consider in \eqref{def:Energ} contains  the essential terms capturing the difficulty of the physical $K_{13}$ problem namely that the energy is unbounded from below. The results we obtain in the next section are relevant to the full physical Oseen-Frank energy, with suitable adaptations.

\section{Constraints on the boundary conditions}
\label{sect:constraints}
    \subsection{The unboundedness of the energy functional}
\label{subsec:unbddenerg}

We show now, by following the example provided in E.Virga's book \cite{Virgabook} and inspired by \cite{OldanoBarbero85} that the energy $\mbbE$ can become unbounded from below, so no global minimizers can exist. Let $\Omega$ be the domain in $\RR^3$ given by $\Omega :=\{(x,y,z): x,y \in (0,l), z \in (-d,+d)\}$ and consider the functions $n_\varepsilon(x,y,z):= (\cos(\rho_\varepsilon(z)), 0 , \sin(\rho_\varepsilon(z))$ with $\rho_\varepsilon(z)$ given by:

\begin{equation*}
\rho_\varepsilon(z) : = \begin{cases}
\rho_0 + \varepsilon - \frac{1 }{\varepsilon^3}(z - d + \varepsilon^2)^2 &\text{ if } z \in [d-\varepsilon^2, d], \\
\rho_0 + \varepsilon &\text{ if }z \in(-d + \varepsilon^2, d-\varepsilon^2), \\
\rho_0 + \varepsilon - \frac{1 }{\varepsilon^3}(z + d - \varepsilon^2)^2 &\text{ if } z \in [-d, -d+\varepsilon^2].
\end{cases}
\end{equation*}A calculation gives that $n_\varepsilon\in \mcA$ and 
\begin{equation*}
\mbbE[n_\varepsilon] : = 4 l^2 \left( \frac{K}{3}-K_{13} \frac{\sin(2\rho_0)}{2\varepsilon} \right).
\end{equation*}
Therefore if we choose $\rho_0$ such that $\sin(\rho_0) >0 $ then we get
\begin{equation*}
\mbbE[n_\varepsilon] \rightarrow -\infty
\end{equation*}
proving that $\mbbE$ is unbounded from below in the function space $\mcA$ (for this specific domain $\Omega$). 

\subsection{Critical points}

Theorem \ref{thm:bdrydet} is proved by studying the first variations of the energy $\mbbE$ in $\mcA$. First we will prove some lemmas; note that in Lemma \ref{lem:critical point equation} it is important that we are working in $W^{2,1}(\Omega,\SSS^2)$.

\begin{lemma}\label{lem:critical point equation}
Let $\bar n$ be a critical point of the energy $\mbbE$ in the function space $\mcA$. Then $\bar n$ satisfies the equation 
\begin{equation*}
\sum_{\alpha,\beta,\gamma =1}^d\int_{\partial \Omega} \left(\varphi_{\beta ,\alpha} \bar n_\alpha - \varphi_{\gamma, \alpha}\bar n_\gamma \bar n_\alpha \bar n_\beta \right) \nu_\beta dS = 0
\end{equation*}
for all $\varphi \in C^\infty (\bar{\Omega} , \RR^3)$ such that $\varphi = 0$ on $\partial \Omega$.
\end{lemma}
\begin{proof}
Let $\bar{n}$ be a critical point of $\mbbE$ in $\mcA$. First, let $\psi \in C^\infty_0(\Omega, \RR^3)$ and set $\bar{n}^{\eps}:= \frac{\bar{n} +\eps \psi}{\abs{\bar{n} + \eps \psi}}$, then a standard calculation gives
\begin{align*}
\frac{d }{d\eps} \mbbE[\bar{n}^\eps] \Big|_{\eps = 0} = K \int_\Omega \nabla \bar{n} \cdot \nabla \psi - \abs{\nabla \bar{n}}^2 n \cdot \psi dx = 0. 
\end{align*}
We now use the fact that $\bar{n} \in W^{2,1}(\Omega, \SSS^2)$ to perform an integration by parts, giving
\begin{equation*}
\frac{d }{d\eps} \mbbE[\bar{n}^\eps] \Big|_{\eps = 0} = K \int_\Omega \left( - \Delta \bar{n}  - \abs{\nabla \bar{n}}^2 \bar{n} \right)\cdot \psi dx = 0.
\end{equation*}
As $\psi$ can be chosen arbitrarily, we get $\Delta \bar{n}  + \abs{\nabla \bar{n}}^2 \bar{n} = 0$ almost everywhere. 

Next, let $\varphi \in C^\infty(\bar{\Omega},\RR^3)$ such that $\varphi=0$ on $\partial\Omega$ and set $\bar{n}^\eps = \frac{\bar{n} +\eps \varphi}{\abs{\bar{n} +\eps \varphi}}$, then a calculation gives
\begin{align*}
\frac{d}{d\eps} \mbbE[\bar{n}^\eps] \Big|_{\eps=0} &= K \int_\Omega \nabla \bar{n} \cdot \nabla \varphi - \abs{\nabla \bar{n}}^2 \bar{n} \cdot \varphi dx + K_{13} \sum_{\alpha,\beta,\gamma =1}^d\int_{\partial \Omega} \left(\varphi_{\beta ,\alpha}  \bar{n}_\alpha - \varphi_{\gamma, \alpha} \bar{n}_\gamma  \bar{n}_\alpha  \bar{n}_\beta \right) \nu_\beta dS \\
&= K \int_\Omega \left(- \Delta \bar{n}  - \abs{\nabla \bar{n}}^2 \bar{n} \right)\cdot \varphi dx + K_{13} \sum_{\alpha,\beta,\gamma =1}^d \int_{\partial \Omega} \left(\varphi_{\beta ,\alpha}  \bar{n}_\alpha - \varphi_{\gamma, \alpha} \bar{n}_\gamma  \bar{n}_\alpha  \bar{n}_\beta \right) \nu_\beta dS \\
&= K_{13} \sum_{\alpha,\beta,\gamma =1}^d \int_{\partial \Omega} \left(\varphi_{\beta ,\alpha}  \bar{n}_\alpha - \varphi_{\gamma, \alpha} \bar{n}_\gamma  \bar{n}_\alpha  \bar{n}_\beta \right) \nu_\beta dS,
\end{align*}
which proves the lemma.
\end{proof}

We now prove an analogue of the fundamental lemma of the Calculus of Variations.

\begin{lemma}\label{lemma:g(x_0) is 0 2d+3d}
Let $d \in \{2,3\}$,  $\Omega \subset \RR^d$ be a $C^1$ domain and $g\in L^\infty(\partial \Omega,\RR)$. Suppose that 
\begin{equation}
	\int_{\partial \Omega} \frac{\partial \varphi}{\partial \nu}g(x) dS(x)= 0
\end{equation}
for all $\varphi \in C^{\infty}(\bar{\Omega},\RR)$ such that $\varphi |_{\partial \Omega} = 0$, where $\nu(x)$ is the unit norm to $\partial \Omega$ at $x$. Then 
$$g(x)=0$$ for almost all points $x\in\partial\Omega$.
\end{lemma}
\begin{proof}
We prove for the case $d=3$, the case $d=2$ is a simpler version. Let $x_0 \in \partial \Omega$ be an arbitrary Lebesgue point of $g$ and let $\psi:(-\varepsilon , \varepsilon)^2 \rightarrow \partial \Omega$ be a coordinate patch such that $\psi(0,0) = x_0$.
If we choose $\varepsilon >0$ sufficiently small then the map 
\begin{align*}
H:[0, \varepsilon)\times(-\varepsilon , \varepsilon)\times(-\varepsilon,\varepsilon)  &\rightarrow \Omega \\
(r,s,t) &\rightarrow [\psi(s,t) - r\nu(s,t)]
\end{align*}
(where $\nu(s,t) = \frac{\psi_s \times \psi_t}{\|\psi_s \times \psi_t\|}$ is the unit normal to $\partial \Omega$ at $\psi(s,t)$) provides a $C^1$ homeomorphism from $(0, \varepsilon)\times(-\varepsilon , \varepsilon)\times(-\varepsilon,\varepsilon)$ onto a relative neighbourhood $U:= H([0, \varepsilon)\times(-\varepsilon , \varepsilon)\times(-\varepsilon,\varepsilon))$ of $x_0$.
For $\delta >0$ small define the maps
\begin{equation*}
\varphi_\delta(x) = \begin{cases} 0 &\text{ if }x \notin U \\
\left( \frac{\varepsilon}{2 \delta}\right)^2 \gamma\left(r,\frac{s \varepsilon}{2 \delta}, \frac{t \varepsilon}{2 \delta}\right) &\text{ if } x \in U,\text{ where } x =  \psi(s,t)-r\nu(s,t). 
\end{cases}
\end{equation*}
where
$$\gamma:[0, \varepsilon)\times(-\varepsilon , \varepsilon)\times(-\varepsilon,\varepsilon)\rightarrow \RR$$
is smooth and $0$ if $(r,s,t) \notin  [0, \frac{\varepsilon}{2}) \times \left(-\frac{\varepsilon}{2},\frac{\varepsilon}{2}\right) \times \left(-\frac{\varepsilon}{2},\frac{\varepsilon}{2}\right)$.
Then we have
\begin{equation*}
	\int_{\partial \Omega}\frac{\partial \varphi_\delta}{\partial \nu}g(x)dx =-\int_{-\delta}^{\delta}\int_{-\delta}^{\delta}\left(\frac{\varepsilon}{2\delta}\right)^2 \frac{\partial \gamma \left(r,\frac{s\varepsilon}{2\delta},\frac{t\varepsilon}{2\delta}\right) }{\partial r} g(\psi(s,t))\|\psi_s \times \psi_t\|\,dsdt = 0.
\end{equation*}
Using the change of variables $\sigma =\frac{s\varepsilon}{2\delta}$ and $\theta = \frac{t\varepsilon}{2\delta} $, we get
\begin{equation}\label{eq:change of variable}
\int_{-\frac{\varepsilon}{2}}^{\frac{\varepsilon}{2}}\int_{-\frac{\varepsilon}{2}}^{\frac{\varepsilon}{2}}\frac{\partial \gamma \left(r,\sigma,\theta\right) }{\partial r} g\left(\psi\left(\frac{2 \delta \sigma}{\varepsilon},\frac{2 \delta \theta}{\varepsilon}\right)\right)\|\psi_s\left(\dfrac{2 \delta \sigma}{\varepsilon},\dfrac{2 \delta \theta}{\varepsilon}\right) \times \psi_t\left(\frac{2 \delta \sigma}{\varepsilon},\dfrac{2 \delta \theta}{\varepsilon}\right)\|d\sigma d\theta = 0.
\end{equation}
On the other hand we have:

\begin{align*}
&\Bigg| \int_{-\frac{\eps}{2}}^{\frac{\varepsilon}{2}} \int_{-\frac{\eps}{2}}^{\frac{\varepsilon}{2}} \frac{\partial \gamma}{\partial r}(0,s,t) \left( g\left(\psi\left( \dfrac{2 \delta s }{\varepsilon},\dfrac{2 \delta t }{\varepsilon}\right) \right) \norm{ \psi_s \times \psi_t \left( \dfrac{2 \delta s }{\varepsilon},\dfrac{2 \delta t }{\varepsilon}\right)}  -   g\left(\psi\left( 0,0 \right) \right) \norm{ \psi_s \times \psi_t \left( 0,0\right)}   \right) \,ds dt \Bigg| \\
 &\leq \Bigg| \int_{-\frac{\eps}{2}}^{\frac{\varepsilon}{2}} \int_{-\frac{\eps}{2}}^{\frac{\varepsilon}{2}} \frac{\partial \gamma}{\partial r}(0,s,t) \norm{ \psi_s \times \psi_t \left( \dfrac{2 \delta s }{\varepsilon},\dfrac{2 \delta t }{\varepsilon}\right)} \left( g\left(\psi\left( \dfrac{2 \delta s }{\varepsilon},\dfrac{2 \delta t }{\varepsilon}\right) \right)  - g\left(\psi\left( 0,0 \right) \right) \right)\, dtds\Bigg| \\
 &\quad + \Bigg| \int_{-\frac{\eps}{2}}^{\frac{\varepsilon}{2}} \int_{-\frac{\eps}{2}}^{\frac{\varepsilon}{2}} \frac{\partial \gamma}{\partial r}(0,s,t) g\left(\psi\left( 0,0\right)\right) \left(   \norm{ \psi_s \times \psi_t \left( \dfrac{2 \delta s }{\varepsilon},\dfrac{2 \delta t }{\varepsilon}\right)}  -    \norm{ \psi_s \times \psi_t \left( 0,0\right)}   \right)\,dtds\Bigg| \\
 &:= I + II   
\end{align*}
Using Cauchy- Schwartz inequality, a change of variables and the fact that $x_0 = \psi(0,0)$ is a Lebesgue point, we have
\begin{align*}
I &\leq \left( \int_{-\frac{\eps}{2}}^{\frac{\varepsilon}{2}} \int_{-\frac{\eps}{2}}^{\frac{\varepsilon}{2}} \left( \frac{\partial \gamma}{\partial r}(0,s,t) \right)^2 \norm{ \psi_s \times \psi_t \left( \dfrac{2 \delta s }{\varepsilon},\dfrac{2 \delta t }{\varepsilon}\right)}\, dtds \right)^{\frac{1}{2}} \\
&\times \left( \int_{-\frac{\eps}{2}}^{\frac{\varepsilon}{2}} \int_{-\frac{\eps}{2}}^{\frac{\varepsilon}{2}} \left( g\left(\psi\left( \dfrac{2 \delta s }{\varepsilon},\dfrac{2 \delta t }{\varepsilon}\right) \right)  - g\left(\psi\left( 0,0 \right) \right) \right)^2 \norm{ \psi_s \times \psi_t \left( \dfrac{2 \delta s }{\varepsilon},\dfrac{2 \delta t }{\varepsilon}\right)}\, dtds \right)^{\frac{1}{2}} \\
&= \left( \int_{-\frac{\eps}{2}}^{\frac{\varepsilon}{2}} \int_{-\frac{\eps}{2}}^{\frac{\varepsilon}{2}} \left( \frac{\partial \gamma}{\partial r}(0,s,t) \right)^2 \norm{ \psi_s \times \psi_t \left( \dfrac{2 \delta s }{\varepsilon},\dfrac{2 \delta t }{\varepsilon}\right)}\, dtds \right)^{\frac{1}{2}} \\
&\times \left(\frac{\varepsilon^2}{4 \delta^2} \int_{-\delta}^{\delta} \int_{-\delta}^{\delta} \left( g\left(\psi\left( \sigma,\tau \right) \right)  - g\left(\psi\left( 0,0 \right) \right) \right)^2 \norm{ \psi_s \times \psi_t \left( \sigma,\tau \right)}\, dtds \right)^{\frac{1}{2}} \rightarrow 0 \text{ as } \delta \rightarrow 0.
\end{align*}
Next, as  $\psi_t,\psi_s$ and $\frac{\partial \gamma }{\partial r}$ are continuous functions, we can use dominated convergence theorem to yield
\begin{align*}
II &=\abs{ \int_{-\frac{\eps}{2}}^{\frac{\varepsilon}{2}} \int_{-\frac{\eps}{2}}^{\frac{\varepsilon}{2}} \frac{\partial \gamma}{\partial r}(0,s,t) g\left(\psi\left( 0,0\right)\right) \left(   \norm{ \psi_s \times \psi_t \left( \dfrac{2 \delta s }{\varepsilon},\dfrac{2 \delta t }{\varepsilon}\right)}  -    \norm{ \psi_s \times \psi_t \left( 0,0\right)}   \right)\,dtds} \\ &\rightarrow 0 \text{ as } \delta \rightarrow 0.
\end{align*} 
We therefore have 
\begin{equation*}
 g(x_0) \norm{\psi_s \times \psi_t (0,0)} \int_{-\frac{\varepsilon}{2}}^{\frac{\varepsilon}{2}} \int_{-\frac{\varepsilon}{2}}^{\frac{\varepsilon}{2}} \frac{\partial \gamma}{\partial r}(0,s,t)\, dsdt = \lim_{\delta \rightarrow 0} \int_{\partial \Omega}\frac{\partial \varphi_\delta}{\partial \nu}g(x)\,dx = 0
\end{equation*}
as $\delta \rightarrow 0$.


Choosing $\gamma(r,s,t) = r\cdot a(s)b(t)$, where $a, b:(-\varepsilon,\varepsilon) \rightarrow \RR$ are smooth and satisfy

\begin{itemize}
\item $a(s),b(t)\geq 0$,
\item $a(s),b(t) >0$ for $s,t \in (-\frac{\varepsilon}{3},\frac{\varepsilon}{3})$,
\item $a(s),b(t) = 0$ for $s,t \notin (-\frac{\varepsilon}{2},\frac{\varepsilon}{2})$.
\end{itemize} 
Then we have
\begin{equation*}
\int_{-\frac{\varepsilon}{2}}^{\frac{\varepsilon}{2}}\int_{-\frac{\varepsilon}{2}}^{\frac{\varepsilon}{2}} \frac{\partial \gamma}{\partial r}(0,s,t) \,ds dt= \int_{-\frac{\varepsilon}{2}}^{\frac{\varepsilon}{2}} \int_{-\frac{\varepsilon}{2}}^{\frac{\varepsilon}{2}}a(s)b(t)\,dsdt >0
\end{equation*} 
and $g(x_0) = 0$ as required.
\end{proof}

\bigskip We can now proceed with the proof of Theorem~\ref{thm:bdrydet}.

\begin{proof}{\bf [of Theorem~\ref{thm:bdrydet}]}
First we consider $d=2$. Let $\bar{n}$ be a critical point of $\mbbE$ in $\mcA$. By lemma~\ref{lem:critical point equation} $\bar{n}$ satisfies the equation
\begin{equation}\label{eq : e-l}
 \int_{\partial \Omega} \sum_{\beta , \alpha = 1}^d \varphi_{\beta , \alpha}\bar{n}_\alpha \left(\nu_\beta - \bar{n}_\beta \langle \bar{n} ,\nu \rangle \right) dS=0,
\end{equation}
for all $\varphi \in C^\infty(\Omega,\RR^3)$ such that $\varphi=0$ on $\partial \Omega$.
Let $\nu(x) = (\nu_1(x) ,\nu_2(x))$ be the unit vector to $\partial \Omega$ at $x$ and $\tau(x)$ be a unit tangent to $\partial \Omega$ at $x$. Then at $x \in \partial \Omega$ we have 
\begin{equation}
\bar{n} = \langle \bar{n}, \nu \rangle \nu + \langle \bar{n} ,\tau \rangle \tau.
\end{equation}
plugging this into \eqref{eq : e-l} gives
$$0 = \int_{\partial \Omega}\sum_{\beta , \alpha = 1}^d \varphi_{\beta , \alpha}\left(\langle \bar{n}, \nu \rangle \nu_\alpha + \langle \bar{n} ,\tau \rangle \tau_\alpha\right) \left(\nu_\beta - \bar{n}_\beta \langle \bar{n} ,\nu \rangle \right) dS.$$
Since $\varphi = 0$ on $\partial \Omega$ we have $\frac{\partial \varphi_\beta}{\partial \tau} = \varphi_{\beta , \alpha} \cdot \tau_\alpha = 0$. Hence \eqref{eq : e-l} simplifies to
\begin{equation}\label{eq: e-l 2}
0 = \int_{\partial \Omega} \sum_{\beta  = 1}^d \frac{\partial \varphi_{\beta}}{\partial \nu}\langle \bar{n}, \nu \rangle  \left(\nu_\beta - \bar{n}_\beta \langle \bar{n} ,\nu \rangle \right) dS.
\end{equation}

If $d=3$ by choosing unit vector fields $P,Q$ such that $P\times Q = \nu$ and $\langle P,Q\rangle = 0$, then by writing
$$\bar{n} = \langle \bar{n}, \nu \rangle \nu + \langle \bar{n} ,P \rangle P +\langle \bar{n} , Q \rangle Q $$
and substituting this into \eqref{eq : e-l} we get that \eqref{eq: e-l 2} holds for $d=3$.
By setting $\varphi_i \equiv 0$  for $i\neq \beta$ in \eqref{eq: e-l 2}, we get
$$\int_{\partial \Omega} \frac{\partial \varphi}{\partial \nu}\langle \bar{n}, \nu \rangle  \left(\nu_\beta - \bar{n}_\beta \langle \bar{n} ,\nu \rangle \right) dS = 0 $$
for $\beta = 1,2,3$ and for all $\varphi \in C^\infty (\Omega , \RR)$ such that $\varphi|_{\partial \Omega} = 0$. By Lemma \ref{lemma:g(x_0) is 0 2d+3d}, we conclude that
\begin{equation}
\langle \bar{n}, \nu \rangle  \left(\nu_\beta - \bar{n}_\beta \langle \bar{n} ,\nu \rangle \right) = 0 \text{ for } \beta = 1,\ldots,d.
\end{equation} 
If $\langle \bar{n} ,\nu \rangle =0$ then we are done. Suppose $\langle \bar{n} ,\nu \rangle \neq 0$, then we must have
$$\left(\nu_\beta - \bar{n}_\beta \langle \bar{n} ,\nu \rangle \right) = 0 \text{ for } \beta = 1,\ldots,d. $$
Since $\langle \bar{n} ,\nu \rangle \neq 0$ 
$$\bar{n}_\beta = \frac{\nu_\beta}{\langle \bar{n}, \nu \rangle} \text{ for } \beta = 1 \ldots d,$$
which implies
\begin{align*}\langle \bar{n} , \nu \rangle &= \sum_{\beta =1}^d \frac{\nu_\beta}{\langle \bar{n} , \nu\rangle}\nu_\beta \\
&=  \frac{\langle \nu , \nu \rangle}{\langle \bar{n} , \nu \rangle} \\
&= \frac{1}{\langle \bar{n} , \nu \rangle}.
\end{align*}
Hence
$$\langle \bar{n} , \nu\rangle^2 = 1$$
and therefore
$$\langle \bar{n} , \nu\rangle = \pm 1.$$
\end{proof}


\section{Partial regularity for tangential boundary conditions}
\label{sect:preg}
We restrict from now on our attention to the case when the boundary conditions are tangential, i.e. $n(x)\cdot\nu(x)=0$ for all $x\in\partial\Omega$, where $\nu$ is the outward pointing unit normal.


We note that for any vector $v$ that is tangent to $\partial \Omega$ at $x$ we have
\begin{align*}
\frac{d \langle n(x),\nu(x)\rangle}{dv} := \sum_{\alpha =1}^3\frac{\partial \langle n(x),\nu(x)\rangle}{\partial x^\alpha}v^\alpha =0.
\end{align*}
as $\langle n(x),\nu(x)\rangle$ is constant in the $v$ direction. That is tangential derivatives of $\langle n(x),\nu(x)\rangle$ are zero for all $x \in \partial \Omega$. Hence, as $n(x)$ is a tangent vector to $\partial \Omega$ at $x \in \partial \Omega$, we have 
\begin{align*}
 \sum_{\alpha,\beta=1}^d n^\alpha \frac{\partial n^\beta}{\partial x^\alpha}\nu^\beta &= \sum_{\alpha=1}^d n^\alpha\frac{\partial (n\cdot \nu)}{\partial x^\alpha} - \sum_{\alpha,\beta=1}^d \frac{ \partial \nu^\beta}{\partial x^\alpha} n^\beta n^\alpha \\
 &=  - \sum_{\alpha,\beta=1}^d \frac{ \partial \nu^\beta}{\partial x^\alpha} n^\beta n^\alpha.
\end{align*}
Therefore, for maps with tangential boundary conditions we can write $\mbbE$ as 
\begin{equation*}
\mathbb{G}[n]:= \int_\Omega K\sum_{\alpha,\beta=1}^d \frac{\partial n^\alpha}{\partial x^\beta}\frac{\partial n^\alpha}{\partial x^\beta} - K_{13}\int_{\partial \Omega} \sum_{\alpha,\beta=1}^d \frac{ \partial \nu^\beta}{\partial x^\alpha} n^\beta n^\alpha.
\end{equation*}
This energy makes sense for maps in $W^{1,2}(\Omega,\SSS^2)$, and so we now focus on the slightly simpler task of minimizing $\mathbb{G}$ over the function space
\be\label{def:mcU+}
\mcU :=\{u \in W^{1,2}(\Omega, \SSS^2): {\rm Trace}(u) \in \mathcal{T} \} 
\ee
where 
\be\label{def:mcT+}
\mathcal{T}:= \{\gamma \in H^{\frac{1}{2}}(\partial \Omega , \SSS^2): \gamma(x)\cdot \nu(x) = 0 \text{ for almost every } x \in \partial \Omega\}.
\ee

Given the topological constraints associated with having tangential boundary conditions the first issue is to show that the function space $\mcU$ is non-empty. This will be addressed in the next subsection, while in the last subsection we will prove a partial regularity result for the minimizers.

\subsection{Function Space is non empty}
In this section we consider $\Omega$ to be a bounded domain of class $C^2$ and study  whether or not the function space $\mathcal{U}$ defined through \eqref{def:mcU+},\eqref{def:mcT+} is non-empty. 

If, for instance, $\partial\Omega$ is the torus then there exist smooth maps in $\mathcal{T}$ that have smooth extensions to the solid torus $\BB\times \SSS^1$ and hence $\mathcal{U}$ would be non empty. However, if $\partial \Omega$ is  $\SSS^2$ then the ``Hairy Ball Theorem" tells us that there are no continuous maps in $\mathcal{T}$ and so it is not immediate that $\mcU$ is non empty. Fortunately, since $H^{\frac{1}{2}}$  in dimension two is larger than the space of continuous functions we are able to show that $\mathcal{T}$ and $\mathcal{U}$ are still non-empty even when $\partial \Omega$ a general $C^2$ surface. \footnote{We just need $C^2$ regularity for using Theorem~\ref{thm:extension}, for all the other results of the section it would suffice to have $C^1$.}

To this end we use an extension Theorem from \cite{HardtLin87} (stated as  Theorem~\ref{thm:extension} in the following) which tells us that a function in $\mathcal{T}$ can be extended to a function in $\mathcal{U}$. This means to show $\mathcal{U}$ is non empty we only need to show that $\mathcal{T}$ is non empty. To do this we construct a function that belongs to $\mathcal{T}$ through a sequence of lemmas. We remark that a map $\gamma \in H^{\frac{1}{2}}(\partial \Omega,\SSS^2)$ is in $\mathcal{T}$ if and only if $\gamma(x) \in T_x\partial \Omega$ for almost every $x\in \partial \Omega$, where $T_x\partial \Omega$ is the tangent space to $\partial \Omega $ at $x$.

In Lemma \ref{lemma:vector field extention lemma}  below we will give necessary and sufficient conditions to extend a vector field from the boundary of a manifold $N$ to its interior. Before we can state Lemma \ref{lemma:vector field extention lemma} we must first introduce some notation:

 If $U$ is a $C^2$ manifold embedded in $\RR^d$, let $T_xU$ be the tangent space to $U$ at $x \in U$. Let $g$ be a smooth vector field on $U$, i.e a smooth map $g:U \rightarrow \RR^d$ such that $g(x) \in T_xU$ for every $x \in U$. Then let $ind(g,U)$ denote the index of $g$ on $U$ (we refer the reader to \cite{Hirschbook},\cite{Spivak_Vol1} or \cite{Canevari2014} 
for detailed properties of $ind(g,U)$).

If $U$ is a manifold with boundary we define
\begin{equation*}
 \partial \_U[g]:= \{x \in \partial U: g(x) \cdot \nu(x) <0\},
\end{equation*}
where $\nu(x)$ is the outward-pointed unit normal to $U$ at $x$.

 Furthermore, we recall that if $U$ is a compact surface  then its Euler characteristic $\chi(U)$ can be related to its topological genus $k$ through the formula
\begin{equation*}
\chi(U) = 2(1-k).
\end{equation*}

\begin{lemma}\label{lemma:vector field extention lemma}
Let $N$ be a $C^1$ manifold with boundary embedded in  $ \RR^d$  and $g \in C^{\infty}(\partial N, \RR^d)$ such that 
\begin{equation}
g(x)  \in T_x(\partial N) \text{ and } \abs{g(x)}=1.\label{eq:conditions on g}
\end{equation}

Then $g$ admits an extension to a continuous field $V:N \rightarrow \RR^d$ such that, for every $x \in N$, $V(x) \in T_xN$, $\abs{V(x)}=1$  and $\left.V\right|_{\partial N} =g$  if and only if 
\begin{equation}\label{eq:condition on N}
\text{ind}(g,\partial\_N[g]) = \chi(N).
\end{equation}

\end{lemma}
\begin{proof}
Let $g \in C^{\infty}(\partial N, \RR^d)$ such that \eqref{eq:conditions on g} and \eqref{eq:condition on N} hold. Let $X$ be the topological double of $N$, that is, the manifold obtained by glueing two copies of $N$ along their boundaries (see \cite{Lee_topManifolds} example 3.80 for a detailed construction of $X$). By modifying the value of $d$ if needed we can assume that $X$ is embedded in  $\RR^d$.

Let $U\subset X$ be a tubular neighbourhood of $\partial N$ such that the nearest point projection $\pi:U \rightarrow \partial N$ is well defined. Let $\varphi:X \rightarrow \RR$ be a smooth function such that $\left.\varphi\right|_{\partial N}=1$ and $\left. \varphi\right|_{X\setminus U}=0$. Then let $\tilde{G}:X\rightarrow \RR^d$ be the extension of $g$ defined by
\begin{equation*}
\tilde{G}(x):= \begin{cases}
Proj_{T_xX}\left(g(\pi(x))\right)\varphi(x) \text{ for } x \in X \cap U, \\
0 \text{ for } x \in X \setminus U,
\end{cases}
\end{equation*}
where $Proj_{T_xX}(y)$ is the projection of $y$ onto $T_xX$. As $0 \notin  \tilde{G}(\partial N)$, by the Transitivity Theorem (see \cite{BrockerJanichBook} Theorem 14.6), there exists a smooth tangent vector field $F$ on $X$ such that $F$ has finitely many zeros in $N$ and $\left.F\right|_{\partial N}=g$. Define $P_{\partial N}F$ to be the map
\begin{equation*}
P_{\partial N}F(x) := Proj_{T_x\partial N}(F(x)) \text{ for } x \in \partial N
\end{equation*}
and for a continuous vector field $v:N\rightarrow \RR^d$, define $\partial \_N[v]$ is to be the set
 \begin{equation*}
 \partial \_N[v]:= \{x \in \partial N: v(x) \cdot \nu(x) <0\} .
 \end{equation*}
By Morse's Index Formula (see \cite{Morse1929}) and the stability of the index we have 
\begin{align*}
ind(F,N) &= ind(\tilde{G},N) \\
&= \chi(N) - ind(P_{\partial N}\tilde{G},\partial\_N[P_{\partial N}\tilde{G}]) \\
&= \chi(N) - ind(g,\partial\_ N[g]) \\
&= 0.
\end{align*}

We now just need to modify $F$ such that $\abs{F}>0$ on $N$. Up to a continuous transformation, we can assume that all the zeros are contained in one coordinate patch $D\subset N$ , with chart $\phi:\overline{D} \rightarrow \overline{B^d(0,1)}$ such that $\phi(\partial D)= \partial B^d(0,1)$, where $B^d(0,1)$ is the ball in $\RR^d$ centred at $0$ with radius 1.  Let $\tilde{F}:B^d(0,1) \rightarrow \RR^d$ be the map defined as 
\begin{equation*}
\tilde{F}(x):= F(\phi^{-1}(x))
\end{equation*}
 and assume that $|\tilde{F}|>0$ in $B^d(0,1) \setminus B^d(0,{\frac{1}{2}})$. Then,
\begin{equation*}
0 = ind(\tilde{F},B^d(0,1)) = deg\left(\frac{F}{|F|},\partial D ; \SSS^{d-1}\right).
\end{equation*} 
It can now be shown, as proved in \cite{HelienWood08}, that there exists a harmonic field $\psi:B^d(0,1) \rightarrow  \SSS^{d-1}$ such that $\left.\psi\right|_{\partial B^d(0,1)} = \frac{\tilde{F}}{|\tilde{F}|}$.  Finally, we define our extension:\begin{align*}
V(x) = \begin{cases}
\frac{F(x)}{\abs{F(x)}} \text{ if } x \in N \setminus D, \\
\psi(\phi(x)) \text{ if } x \in D.
\end{cases}
\end{align*}
$V$ is continuous and smooth everywhere apart from $\partial D$.
\end{proof}
\textbf{Remark}: Note that in the above construction we get a vector field on $N$ that is smooth almost everywhere.

\bigskip
We now  relate the $H^{\frac{1}{2}}$ to a space whose norm is easier to compute, the $W^{1,p}$ space:

\begin{lemma}\label{lemma: W^1,p into H^1/2}
Let $U \subset \RR^2$ be an open set with $C^1$ boundary and $u \in W^{1,p}(U,\SSS^2)$ for $1<p<2$. Then $u \in H^{\frac{1}{2}}(U , \SSS^2)$.
\end{lemma}

In order to prove this we use the following Propositions from \cite{HitchhikersSobSpace} :
\begin{theorem}\label{thm: frac sob embed}
Let $p \in [1,\infty)$ and $s \in (0,1)$. Let $\Omega$ be an open set in $\RR^n$ of class $C^{1}$ with bounded boundary and $u\in W^{1,p}(\Omega,\RR)$. Then 
$$\norm{u}_{W^{s,p}(\Omega)} \leq C \norm{u}_{W^{1,p}(\Omega)}$$
for some positive constant $C= C(n,s,p)\geq 1$.
\end{theorem}
and
\begin{theorem}\label{thm:frac sob extension lemma}
Let $p \in [1,+\infty)$, $s \in (0,1)$ and $\Omega \subseteq \RR^n$ be an open set of class $C^{1}$ with bounded boundary. Then $W^{s,p}(\Omega)$ is continuously embedded in $W^{s,p}(\RR^n)$, namely there exists $C=C(n,\Omega)$ such that for any $u \in W^{s,p}(\Omega)$ there exists $\tilde{u} \in W^{s,p}(\RR^n)$ such that $\tilde u\big|_\Omega = u$ and
\begin{equation*}
\norm{\tilde{u}}_{W^{s,p}(\RR^n)}\leq C \norm{u}_{W^{s,p}(\Omega)}.
\end{equation*} 
\end{theorem}
We also use the interpolation lemma from \cite{Mazya02}
\begin{theorem} \label{thm: fractional sobolev spaces interpolation}
For all $u \in W^{s,q}(\RR^n)\cap L^\infty(\RR^n)$ there holds the inequality
\begin{equation*}
\norm{u}_{W^{\theta s,q/\theta}(\RR^n)} \leq c(n)\left(\frac{q}{q-1}\right)^\theta \left( \frac{1-s}{1-\theta}\right)^{\frac{\theta}{q}}\norm{u}^\theta_{W^{s,q}(\RR^n)}\norm{u}^{1-\theta}_{L^\infty}
\end{equation*}
where $0<s<1,1<q<\infty$, and $0<\theta<1$.
\end{theorem}
We can now prove Lemma \ref{lemma: W^1,p into H^1/2}.	
\begin{proof}{\bf [Lemma \ref{lemma: W^1,p into H^1/2}]}
Let $1<p<2$ and $u \in W^{1,p}(U)$. By Theorem \ref{thm: frac sob embed} we have that $u \in W^{\frac{2}{3},p}(U)$. Let $\tilde{u} \in W^{\frac{2}{3},p}(\RR^2)$ be the extension given by Theorem \ref{thm:frac sob extension lemma}. By setting $s= \frac{2}{3},\theta=\frac{3}{4}$ and $q = \frac{3}{2}$ in Theorem \ref{thm: fractional sobolev spaces interpolation} we have
\begin{align*}
\norm{u}_{H^{\frac{1}{2}}(U)} &\leq \norm{\tilde{u}}_{H^{\frac{1}{2}}(\RR^2)} = \norm{\tilde{u}}_{W^{\frac{2}{3}\cdot\frac{3}{4}, 2}(\RR^2)}\leq C \norm{\tilde{u}}_{W^{\frac{2}{3}, \frac{3}{2}}(\RR^2)}^{\frac{3}{4}} \norm{u}_{L^{\infty}(\RR^2)}^{\frac{1}{4}}\\
&\leq C \norm{u}_{W^{\frac{2}{3}, \frac{3}{2}}(U)}^{\frac{3}{4}} \norm{u}_{L^{\infty}(\RR^2)}^{\frac{1}{4}} \leq C \norm{u}_{W^{1, \frac{3}{2}}(U)}^{\frac{3}{4}} \norm{u}_{L^{\infty}(\RR^2)}^{\frac{1}{4}}<\infty. 
\end{align*}
\end{proof}

In order to prove Proposition \ref{prop:nonemptyfspace} we use an extension theorem from \cite{HardtLin87}, which we state here for the reader's convenience. 

\begin{theorem}\label{thm:extension}
If $1<p<m$, $\Omega$ is a bounded $C^2$ domain in $\RR^m$, and $N$ is a compact $C^2$ submanifold of $\RR^k$ with $\pi_0(N)=\pi_1(N)= \dots \pi_{[p]-1}(N)=0$, then any function $\eta \in W^{1-1/p,p}(\partial\Omega,N)$ admits an extension $\omega \in W^{1,p}(\Omega, N)$.\label{thm:W^1-1/p,p extension thm}
\end{theorem}

We are now in a position to prove Proposition \ref{prop:nonemptyfspace}
\begin{proof}{[\bf Proposition \ref{prop:nonemptyfspace}]}
Let $\Omega$ be a $C^2$ domain such that $\partial \Omega$ is a surface of genus $k$. Then we have that
\begin{equation*}
\chi(\partial \Omega)=2(1-k).
\end{equation*}
We first show that the map $\varphi:B^2(0,1)\rightarrow \SSS^1$ defined by
\begin{equation*}
\varphi(x^1,x^2) = \left( \frac{-x^2}{\sqrt{(x^1)^2 +(x^2)^2}} , \frac{x^1}{\sqrt{(x^1)^2 +(x^2)^2}} \right)
\end{equation*}
is in $W^{1,p}(B^2(0,1),\SSS^1)$. Then $|\nabla\varphi|=\frac{1}{|x|}$ and we have:
thus 
\begin{align*}
\int_{B^2(0,1)}\abs{\nabla \varphi}^p dx &\leq C\int_0^1 r^{1-p} dr
\end{align*}
which is finite for $1 < p < 2$, hence $\varphi \in W^{1,p}(B^2(0,1),\SSS^1)$ for $1<p< 2$.

For $i=1 \dots |\chi(\partial \Omega)|$ let $(U_i,\psi_i)$ be coordinate patches on $\partial \Omega$ such that $\bigcap_{i=1}^{|\chi(\partial \Omega)|} U_i = \emptyset$.
Using
\begin{equation*}
\varphi(x^1,x^2) = \left( \frac{-x^2}{\sqrt{(x^1)^2 +(x^2)^2}} , \frac{x^1}{\sqrt{(x^1)^2 +(x^2)^2}} \right),
\end{equation*}
or
\begin{equation*}
 \tilde{\varphi}(x^1,x^2) = \left( \frac{x^2}{\sqrt{(x^1)^2 +(x^2)^2}} , \frac{-x^1}{\sqrt{(x^1)^2 +(x^2)^2}} \right)
\end{equation*}
we can put a unit vector field, $v_i$, on each $U_i$ such that the map $V : \bigcup U_i \rightarrow \SSS^2$ defined by
\begin{equation*}
v(x):= v_i(x) \text{ if } x \in U_i
\end{equation*} 
satisfies
\begin{equation*}
 ind\left(v\big|_{\partial \left(\cup U_i\right)} , \partial\_(\partial \Omega \setminus \cup U_i))[v]\right) = \chi(\partial \Omega).
\end{equation*}
Using Lemma \ref{lemma:vector field extention lemma} we can now extend $V$ to a vector field on $\partial \Omega$ that is smooth on $\partial \Omega \setminus \bigcup U_i$. Since $\varphi$ and $\tilde{\varphi}$ are both in $W^{1,p}( \partial \Omega,\SSS^2)$ for $1<p<2$ the map $V$ is in $W^{1,p}(\partial \Omega,\SSS^2)$. Then by Lemma \ref{lemma: W^1,p into H^1/2} we have that $V \in H^{\frac{1}{2}}(\partial \Omega,\SSS^2)$, it follows that $V \in \mathcal{T}$ and hence $\mathcal{T}$ is non empty. Using  Theorem \ref{thm:W^1-1/p,p extension thm} we can now conclude that the function space $\mathcal{T}$ is non empty.
\end{proof}
\subsection{Regularity of Minimizers}
We prove that a minimizer, $u$, of $\mathbb{G}$ in $\mathcal{U}$ is continuous on $\Omega\setminus Z$, where $Z$ is some subset of $\Omega$ with zero one-dimensional Hausdorff measure. We show that for $x \in \Omega\setminus Z$ the {\it rescaled energy}
\begin{equation*}
\mbbE_r[u]:=r^{-1} \int_{\Omega \cap C(a,r)} \abs{\nabla u}^2dx
\end{equation*}
decays suitably fast as $r$ tends to $0$. Continuity of $u$ then follows by Morrey's Lemma (see for example \cite{taheri2015function} Chapter 18). Our proof of the energy decay is based on the work of Schoen and Uhlenbeck \cite{UhlenbeckSchoen82} and the papers by Hardt and Lin \cite{HardtLin87}, \cite{HardtLin89}. 

We give a brief outline of the proof here. For a given domain $\Omega$ we construct a map $Q:\Omega \rightarrow SO(3)$ (depending only on $\partial\Omega$) such that for an arbitrary map $u \in \mathcal{U}$ we have that $Q(u(x))$ lies on the equator of $\SSS^2$ for almost all $x \in \partial \Omega$. This transforms our tangential boundary conditions into a partially constrained boundary condition as in \cite{HardtLin89}, so we can adapt some of the results there in order to prove a Hybrid inequality (Lemma \ref{lem:Hybrid Inequality}). The proof of the energy decay (Lemma \ref{lem:Energy Improvement} and Theorem \ref{Thm: energy decay}) is then done by contradiction. We assume there exists a sequence of minimizers, $u_i$, and domains, $\Omega_i$, that do not have energy decay but do satisfy $\varepsilon_i= \int_{\Omega_i} \abs{\nabla u_i}^2\,dx \rightarrow 0$ as $i \rightarrow \infty$. We then form the blow-up sequence $\varepsilon_i^{-1}(u_i- \dbar{u}_i)$ that converges weakly to a blow up function $v$. We then prove some estimates on $v$ by showing it is harmonic. These estimates are then transferred to the $u_i$ using the Hybrid inequality in order to get a contradiction.

\subsubsection{Scaling and notations}\label{subsec: scaling}
In order to apply Morrey's Lemma we must investigate how our energy scales. For $x \in \RR^3$ define the cylinder $C(x,r):= \{y \in \RR^3:\abs{(y^1,y^2)-(x^1,x^2)}<r , \abs{y^3 -x^3} <r\}$. Let 

\be\label{def:varphi}
\varphi \in C^2(\RR^2,\RR)\textrm{ with }\varphi(0) = 0 = \abs{\nabla \varphi(0)}, \,{\rm Lip}(\varphi)\le 1
\ee and define 
\begin{align*} 
\Omega_{\varphi}:&= \{(x^1,x^2,x^3) \in C(0,1) : x^3 < \varphi(x^1,x^2)\} \\
&= \{ x \in \RR^3 : \abs{(x^1,x^2)}<1 \text{ and } -1< x^3 < \varphi(x^1,x^2)\}.
\end{align*}

For a domain $\Omega\subset \RR^3, a \in \partial \Omega$ there exists $R>0$, $h \in SO(3)$ and $\varphi_{R,a} \in C^2(\RR^2,\RR)$ such that $\varphi_{R,a}(0)=0=\abs{\nabla \varphi_{R,a}(0)} , {\rm Lip}(\varphi_{R,a}) \leq 1$ and 
\begin{equation*}
\Omega_{\varphi_{R,a}} = \{h^{-1}[(y-a)/R] : y \in C(a,R) \cap \Omega\}.
\end{equation*} 
For $0<r \leq R$, let $\varphi_{r,a} = \varphi_{R,a}(\frac{rx}{R})$, then for $u \in \mcU$, the expression $n_{r,a}(x) = n[rh(x)+a]$ defines a function in $W^{1,2}(\Omega_{\varphi_{r,a}},\SSS^2)$ whose trace, $\gamma$, on $\partial \Omega_{\varphi_{r,a}}\setminus \partial C(0,1)$ satisfies $\gamma(x)\cdot (h\cdot\nu(x))+a) = 0$ almost everywhere. We note that 

 $$\mbbE_1(u_{r,a}) = \int_{\Omega_{\varphi}\cap C(0,1)}\abs{\nabla u_{r,a}(x)}^2 dx = r^{-1}\int_{\Omega_{\varphi} \cap C(a,r)} \abs{\nabla u}^2 dx$$

$${\rm Lip}(\varphi_{r,a})=\lip{\varphi_{R,a}}R^{-1}r$$

$$\liptwo{\varphi_{r,a}}:=\max_{\abs{\alpha} = 2} \sup_{x \in \RR^2}\abs{\frac{\partial^{|\alpha|} \varphi_{r,a}(x)}{\partial x^{\alpha}}} = R^{-2}r^2 \norm{\varphi_{R,a}}_{C^2}. $$
For convenience  we collect the notations of various domains  we will use
$$
\begin{array}{ll}
\Omega_{\varphi} := \{(x^1,x^2,x^3)\in C(0,1) : x^3 < \varphi(x^1,x^2)\} & \Omega_0 := \{(x^1,x^2,x^3)\in C(0,1) : x^3 < 0\} \\
 G_{\varphi} := \partial \Omega_{\varphi} \setminus \partial C(0,1) &
 G_{0} := \partial \Omega_{0} \setminus \partial C(0,1) \\ H_{\varphi} := \partial \Omega_{\varphi} \cap \partial C(0,1) &
H_{0} := \partial \Omega_{0} \cap \partial C(0,1) \\
\BB_r:=\{x\in\R^3; |x|\le r\} & \BB:=B^3(0,1)
\end{array}
$$
Using these notations we define the energy, $\mathbb{G}_\varphi$, for a map  $u \in W^{1,2}(\Omega_\varphi,\SSS^2)$   to be
\begin{equation*}
\mathbb{G}_{\varphi}[u] : = \int_{\Omega_{\varphi}} \abs{\nabla u}^2 \, dx - K_{13} \int_{G_\varphi} u^\alpha u^\beta \frac{\partial \nu^\beta }{\partial x^\alpha}\, dx.
\end{equation*}
\subsubsection{A Useful Projection}

Given $\varphi$ as above we will construct a map
$$Q:\Omega_{\varphi} \rightarrow SO(3)$$
such that if 
$$u \in \mcU_{\varphi}:=\{u \in W^{1,2}(\Omega_{\varphi},  \SSS^2) : {\rm Trace}(u) \in \mathcal{T}_{\varphi}\}$$
where 
$$\mathcal{T}_{\varphi} = \{\gamma \in H^{\frac{1}{2}}(G_{\varphi},\SSS^2): \gamma(x)\cdot \nu(x) = 0 \text{ almost everywhere }\}.$$
then the map $w(x):= Q(x)u(x)$ is in the function space
$$\mcE_{\varphi}=\{v \in W^{1,2} (\Omega_{\varphi}, \SSS^2) : {\rm Trace}(v) \in \tilde{\SSS}^1\},$$
where $\tilde{\SSS}^1 = \{(x^1,x^2,x^3)\in \SSS^2 : x^3= 0 \}$.

\bigskip We construct $Q$ as follows: let $x \in G_\varphi$ and let $\nu(x)$ be the unit norm to $G_\varphi$ at $x$. Define $Q(x)$ to be the rotation about the axis $\nu(x) \times (0,0,1)$ through the angle $\tau$ given by $\cos(\tau)= \nu(x)\cdot (0,0,1)$. Explicitly $Q$ is given by 

\begin{equation} \label{def: Q def }
Q(x) := \left[ \begin{matrix}
\dfrac{\varphi_{x^2}^2 + \varphi_{x^1}^2 \cos(\tau)}{\varphi_{x^1}^2+\varphi_{x^2}^2} & -\dfrac{\varphi_{x^1}\varphi_{x^2}(1-\cos(\tau))}{\varphi_{x^1}^2+\varphi_{x^2}^2} & \dfrac{\varphi_{x^1} \sin (\tau)}{(\varphi_{x^1}^2 + \varphi_{x^2}^2)^{\frac{1}{2}}} \\
-\dfrac{\varphi_{x^1}\varphi_{x^2}(1-\cos(\tau))}{\varphi_{x^1}^2+\varphi_{x^2}^2} & \dfrac{\varphi_{x^1}^2 + \varphi_{x^2}^2 \cos(\tau)}{\varphi_{x^1}^2+\varphi_{x^2}^2} & \dfrac{\varphi_{x^2} \sin (\tau)}{(\varphi_{x^1}^2 + \varphi_{x^2}^2)^{\frac{1}{2}}} \\
-\dfrac{\varphi_{x^1} \sin (\tau)}{(\varphi_{x^1}^2 + \varphi_{x^2}^2)^{\frac{1}{2}}} & -\dfrac{\varphi_{x^2} \sin (\tau)}{(\varphi_{x^1}^2 + \varphi_{x^2}^2)^{\frac{1}{2}}} & \cos(\tau) 
\end{matrix} \right],
\end{equation}
where $\varphi_{x_i} = \frac{\partial \varphi}{\partial x_i}$, $\cos(\tau)=(\varphi_{x^1}^2 + \varphi_{x^2}^2 + 1)^{-\frac{1}{2}}$ and $\sin(\tau)= \sqrt{\frac{\varphi_{x^1}^2 + \varphi_{x^2}^2}{\varphi_{x^1}^2 + \varphi_{x^2}^2 + 1}}$. Then for $x = (x^1,x^2,x^3) \in \Omega_\varphi$ define
\begin{equation*}
Q(x) = Q\big((x^1,x^2,\varphi(x^1,x^2))\big).
\end{equation*}
It is straightforward to check that $Q \in C^1(\Omega_\varphi, SO(3))$. We now find some bounds on the entries of $Q$ that will be useful throughout our proof.

\begin{lemma}\label{lem: Qn-n bound}
Let $Q, \varphi$ be defined as above, then we have
$$\sup \{\abs{Q(x)n-n}:x \in \Omega_{\varphi}, n \in \SSS^2 \} \leq 9 \lip{\varphi}.$$
\end{lemma}
\begin{proof}
Let $x \in G_\varphi$ and $n\in\SSS^2$. We have $\abs{Q(x)n - n} \leq \sum_{i=1}^3\abs{(Q(x)n-n)_i}$, where $(Q(x)n-n)_i$ is the $i^{th}$ component of $Q(x)n - n$. We find bounds on each component individually.
$$\abs{(Q(x)n-n)_1}= \abs{ \left(\frac{\varphi_{x^2}^2 + \varphi_{x^1}^2\cos(\tau) - \varphi_{x^1}^2 - \varphi_{x^2}^2}{\varphi_{x^1}^2 + \varphi_{x^2}^2}\right)n_1 - \frac{\varphi_{x^1}\varphi_{x^2}(1 - cos(\tau))n_2}{\varphi_{x^1}^2 + \varphi_{x^2}^2} + \varphi_{x_1} \cos(\tau)n_3 } $$
Writing $\varphi_{x^1}(x) = r \cos(t), \varphi_{x^2}(x) = r \sin(t)$ for appropriate $r \in (0,\infty),t \in [0,2\pi)$ we have
\begin{align*}
\abs{(Q(x)n-n)_1} &\leq \abs{\frac{r^2 \cos^2(t)\left( (r^2 +1)^{-\frac{1}{2}} - 1 \right)}{r^2}} + \abs{\frac{r^2\cos(t)\sin(t) \left( 1 - (r^2 +1)^{-\frac{1}{2}} \right)}{r^2} } +\abs{r} \\
&\leq 2r +r \\
&\leq 3\lip{\varphi}
\end{align*}
We can get a similar estimate for $(Q(x)n-n)_2$ and $(Q(x)n-n)_3$, hence
$$\abs{Q(x)n-n} \leq 9 \lip{\varphi}.$$
\end{proof}

\begin{lemma}\label{lem: bound on dQ/dx}
$$\sup \left\{\abs{\frac{\partial Q_{ij}(x)}{\partial x^k}}: i,j,k \in \{1,2,3\},x \in \Omega_\varphi\right\} < 6 \liptwo{\varphi} $$ 
\end{lemma}
\begin{proof}
Let $x \in \Omega_\varphi$, we find bounds on each $\frac{\partial Q_{ij}(x)}{\partial x^k}$, write $\varphi_{x^1}(x) = r \sin(t), \varphi_{x^2}(x)= r \cos(t)$.
\begin{align*}
\abs{\frac{\partial Q_{11}(x)}{\partial x^k}} &= \abs{\frac{\varphi_{x^1}^2\frac{\partial cos(\tau)}{\partial x^k}}{(\varphi_{x^1}^2+\varphi_{x^2}^2)}  + 
2\left[ \frac{\varphi_{x^2 x^k}\varphi_{x^1}^2 \varphi_{x^2}(1 - \cos(\tau)) + \varphi_{x^1 x^\alpha} \varphi_{x^1} \varphi_{x^2}^2(\cos(\tau) - 1) }{(\varphi_{x^1}^2 + \varphi_{x^2}^2)^2} \right]
} \\
&\leq \abs{ \frac{\varphi_{x^1}^2(\varphi_{x^1}\varphi_{x^1x^k} + \varphi_{x^2}\varphi_{x^2x^k})}{(\varphi_{x^1}^2+\varphi_{x^2}^2)(\varphi_{x^1}^2+\varphi_{x^2}^2+1)^\frac{3}{2}}} + 2 \liptwo{\varphi} \left[ \abs{\frac{\sin^2(t)\cos(t)(1 - \cos(\tau))}{r} } \right.  \\ &\qquad \qquad + \left. \abs{ \frac{\sin(t)\cos(t)(1 - cos(\tau))}{r}} \right] \\
&\leq 2 \liptwo{\varphi}\lip{\varphi} + 4 \liptwo{\varphi}\lip{\varphi}
\end{align*}
By the symmetry of $Q$, $\frac{\partial Q_{22}}{\partial x^k}$ will have the same bound. 

Next, bounding $\frac{\partial Q_{12}}{\partial x^k}$;
\begin{align*}
\abs{\frac{\partial Q_{12}}{\partial x^k} } 
&= 
\left| \frac{(\varphi_{x^1 x^k} \varphi_{x^2} + \varphi_{x^1}\varphi_{x^2x^k})(\cos(\tau)-1)}{\varphi_{x^1}^2 + \varphi_{x^2}^2} + 2\frac{\varphi_{x^1}\varphi_{x^2}(\varphi_{x^1}\varphi_{x^1x^k} + \varphi_{x^2}\varphi_{x^2x^k})(1 - cos(\tau))}{(\varphi_{x^1}^2 + \varphi_{x^2}^2)^2} \right.  \\  &\qquad \qquad \left. -  \frac{\varphi_{x^1}\varphi_{x^2}(\varphi_{x^1}\varphi_{x^1x^k} + \varphi_{x^2}\varphi_{x^2x^k})}{(\varphi_{x^1}^2 + \varphi_{x^2}^2)(\varphi_{x^1}^2 + \varphi_{x^2}^2 + 1)^{\frac{3}{2}}} \right| \\
&\leq \liptwo{\varphi}r^2\left( \abs{\frac{r\cos(t)}{r^2}} + \abs{\frac{r\sin(t)}{r^2} } \right) + 2r^4\liptwo{\varphi}\left( \abs{\frac{r\cos(t)}{r^4}} + \abs{\frac{r\sin(t)}{r^4} } \right) \\
&\leq 6\liptwo{\varphi}\lip{\varphi} \\
&\leq 6\liptwo{\varphi}.
\end{align*}
Again, by the symmetry of $Q$, $\frac{\partial Q_{21}}{\partial x^k}$ will have the same bounds.
We have 
\begin{align*}
\abs{\frac{\partial Q_{13}}{\partial x^k} } &= \abs{ \frac{\varphi_{x^1 x^k}}{(\varphi_{x^1}^2 + \varphi_{x^2}^2+1)^{\frac{1}{2}}} - \frac{\varphi_{x^1}(\varphi_{x^1} \varphi_{x^1x^k} + \varphi_{x^2}\varphi_{x^2x^k})}{(\varphi_{x^1}^2 + \varphi_{x^2}^2+1)^{\frac{3}{2}} }} \\
&\leq \liptwo{\varphi} + 2 \liptwo{\varphi}\lip{\varphi}^2 \\ 
&\leq 3\liptwo{\varphi}
\end{align*}
and $\frac{\partial Q_{23}}{\partial x^k}, \frac{\partial Q_{31}}{\partial x^k}, \frac{\partial Q_{32}}{\partial x^k}$ will have the same bounds. Finally,
\begin{align*}
\abs{\frac{\partial Q_{33}}{\partial x^k}} &= \abs{ \frac{\varphi_{x^1}\varphi_{x^1x^k} + \varphi_{x^2}\varphi_{x^2x^k}}{(\varphi_{x^1}^2 + \varphi_{x^2}^2 + 1)^{\frac{3}{2}}} } \\
&\leq 2 \lip{\varphi}\liptwo{\varphi} \\
&\leq 2 \liptwo{\varphi}.
\end{align*}
This gives the result with $K=6$.
\end{proof}

\begin{lemma}\label{lem: bound on Q_ij}
\begin{equation*}
\sup \{ Q_{ij}(x) : i,j \in \{1,2,3\} \text{ and } x \in \Omega_{\varphi} \} = 1
\end{equation*}
\end{lemma}
\begin{proof}
First we note that $Q_{11}(0) = 1$ and so $1 \leq \sup\{ Q_{ij}(x) : i,j \in \{1,2,3\}, x \in \Omega_{\varphi} \}$. Next we calculate
\begin{equation*}
\abs{Q_{11}(x)}\leq \frac{\varphi^2_{x^2}}{\varphi^2_{x^1} + \varphi^2_{x^2}} + \frac{\varphi^2_{x^1}}{\varphi^2_{x^1} + \varphi^2_{x^2}} = 1
\end{equation*} 

\begin{equation*}
\abs{Q_{12}(x)}\leq \abs{ \frac{\varphi_{x^2}\varphi_{x^1}(1 - \cos(\tau))}{\varphi^2_{x^1} + \varphi^2_{x^2}} } \leq \abs{ \frac{2\varphi_{x^2}\varphi_{x^1}}{\varphi^2_{x^1} + \varphi^2_{x^2}} }\leq 1 
\end{equation*} 
\begin{equation*}
\abs{Q_{13}(x)}\leq \abs{\frac{\varphi_{x^1}}{(\varphi^2_{x^1} + \varphi^2_{x^2} + 1)^{\frac{1}{2}}} }\leq \abs{\varphi_{x^1}} \leq 1 
\end{equation*} 
\begin{equation*}
\abs{Q_{33}(x)} = \abs{\cos(\tau)} \leq 1
\end{equation*}
The other terms are similar or the same.
\end{proof}

\subsubsection{Proof of Partial Regularity}

We will use the projection $Q$ and the following extension Lemma from \cite{HardtLin89} in order to prove a Hybrid Inequality. Note that the following lemma is stated for balls rather than cylinders.

\begin{lemma}\label{lem: exten lem}
There are positive constants $\delta,q$ and $c$ such that, if $0 <\varepsilon <1$, $\xi \in \RR^3$, and $\eta \in W^{1,2}(\Omega_{\varphi} \cap \partial \BB,\SSS^2)$ satisfies the small oscillation condition 
\begin{equation}\label{eq:small oscillation}
\left[ \int_{\Omega_{\varphi} \cap \partial \BB}\abs{\nabla_{tan}\eta}^2 d \mcH^2 \right]\left[\int_{\Omega_{\varphi} \cap \partial \BB}\abs{\eta - \xi}^2 d \mcH^2 + \int_{\partial( \Omega_{\varphi} \cap \partial \BB)}\abs{\eta - \xi}^2 d \mcH^1 \right] \leq \delta^2 \varepsilon^q
\end{equation}
and if $\left. \eta\right|_{\partial \Omega_{\varphi} \cap \partial \BB}$ has image in $\SSS^1$, then there exists a function $\omega \in W^{1,2}(\Omega_{\varphi}\cap \BB,\SSS^2)$, $\left.\omega\right|_{\Omega_{\varphi} \cap \partial \BB}= \eta$,$\left. \omega \right|_{G_{\varphi}}$ has image in $\SSS^1$ and 
$$\int_{\Omega_{\varphi}\cap \BB}\abs{\nabla \omega}^2 dx \leq \varepsilon \int_{\Omega_{\varphi} \cap \partial \BB} \abs{\nabla_{tan}\eta }^2 d \mcH^2 + c \varepsilon^{-q}\left[ \int_{\Omega_{\varphi} \cap \partial \BB}\abs{\eta - \xi}^2 d \mcH^2 + \int_{\partial (\Omega_{\varphi} \cap \partial \BB)}\abs{\eta - \xi}^2 d \mcH^1 \right]. $$
\end{lemma}
\begin{proof}
For a proof see Theorem 3.1 in \cite{HardtLin89}.
\end{proof}

\bigskip We can now prove our main ingredient:

\begin{boxlemma}{\bf [Hybrid inequality]}\label{lem:Hybrid Inequality}
There exists positive constants $c_1,c_2$ and $q$ such that if $0 \leq \lambda \leq 1$, $\varphi$ is as above, $u$ is a minimizer of $\mathbb{G}_\varphi$ amongst maps in $\mcU_{\varphi}$ with fixed trace on $H_{\varphi}$ such that $\int_{\Omega_{\varphi}} \abs{\nabla u}^2 \, dx < c_1 \lambda^{\frac{q}{2}}$, $ \lip{\varphi} \leq c_1\lambda^{\frac{q}{2}}$ and $ \liptwo{\varphi} \leq c_1\lambda^{\frac{q}{2}}$, then 
\begin{align*}
\left(\frac{1}{2}\right)^{-1} \int_{\Omega_\varphi \cap C\left(0,\frac{1}{4}\right)}\abs{\nabla u }^2 \, dx \leq &\lambda \left( \liptwo{\varphi}^2+ \int_{\Omega_{\varphi}\cap C(0,1)} \abs{\nabla u }^2\, dx \right) \non\\
&+ \lambda^{-q}c_2 \left( \lip{\varphi}^2 + \int_{\Omega_\varphi\cap C(0,1)}\abs{u -\dbar{u}}^2 \,dx\right)\\ 
& + \lambda^{-q}c_2\int_{G_{\varphi}\cap C(0,1)}\abs{u - \dbar{u}}^2 \, d \mcH^2  +   c_2 \liptwo{\varphi}^2.
\end{align*}   where $\dbar{u} := \mcH^2(G_{\varphi})^{-1} \int_{G_\varphi}u d\mcH^2$.
\end{boxlemma}
\begin{proof}
As we have the set inclusions $C\left( 0,\frac{1}{4}\right)\subset B^3(0,\frac{1}{2})$ and $\BB \subset C(0,1)$ it suffices to prove the inequality
\begin{align*}
\left(\frac{1}{2}\right)^{-1} \int_{\Omega_\varphi \cap \BB_{\frac{1}{2}}}\abs{\nabla u }^2 \, dx \leq &\lambda \left( \liptwo{\varphi}^2
+ \int_{\Omega_{\varphi}\cap \BB} \abs{\nabla u }^2\, dx \right) \\
&+ \lambda^{-q}c_2 \left( \lip{\varphi}^2 + \int_{\Omega_\varphi\cap \BB}\abs{u -\dbar{u}}^2 \,dx\right)\\
&  + \lambda^{-q}c_2 \int_{G_{\varphi}\cap \BB}\abs{u - \dbar{u}}^2 \, d \mcH^2 +   c_2 \liptwo{\varphi}^2.
\end{align*}

Let $u$ be a map satisfying the assumptions of the Lemma. We aim to apply Lemma \ref{lem: exten lem} to the map $Q(x)u(x)$. In order to do this we must check the map $Q(x)u(x)$ satisfies the small osculation condition \eqref{eq:small oscillation}. We first bound
\begin{align}
\abs{\nabla(Q(x)u(x))}^2 &= \sum_{i,k,l = 1}^3 \left( \frac{\partial (Q_{ij}u^j)}{\partial x^k}\right)^2 \notag \\
&= \sum_{i,j,k}\left( \frac{\partial Q_{ij}}{\partial x^k}u^j + \frac{\partial u^j}{\partial x^k}Q_{i,j}\right)^2 \notag \\
&\leq 2 \sum_{i,j,k} \left[ \left(\frac{\partial Q_{ij}}{\partial x^k}u^j\right)^2 + \left( \frac{\partial u^j}{\partial x^k}Q_{ij}\right)^2\right] \notag \\
&\leq 2\sum_{i,k} \left[ \sum_{j=1}^3\left(\frac{\partial Q_{i,j}}{\partial x^k} \right)^2\sum_{l=1}^3(u^l)^2 +\sum_{m=1}^3\left( \frac{\partial u^m}{\partial x^k}\right)^2\sum_{n=1}^3(Q_{in})^2 \right] \notag \\
&\leq 18 \left( \liptwo{\varphi}^2 + \abs{\nabla u }^2 \right) \label{ineq:bound on nabla Qu}
\end{align}
by Lemma \ref{lem: bound on dQ/dx} and \ref{lem: bound on Q_ij}. 

As in \cite{HardtLinKinderlehrer86}, we note that  for an increasing function $\eta:[0,1] \rightarrow \RR$ we have
$$\mathcal{L}^1\left( \{s: \eta'(s) \geq 8(\eta(1) - \eta(0))\} \right) \leq \frac{1}{8}.$$
Hence we can find $\sigma \in \left[\frac{\sqrt{2}}{2},1\right]$ such that $\left. u \right|_{\Omega_{\varphi}\cap \partial B^3(0,\sigma)}\in W^{1,2}(\Omega_{\varphi}\cap\partial B^3(0,\sigma),\SSS^2)$, 
\begin{equation}\label{ineq:hyb 1}
\int_{\Omega_{\varphi}\cap\partial B^3(0,\sigma)}\abs{Qu-\dbar{u}}^2 \, d \mcH^2 + \int_{\partial \left(B^3(0,\sigma) \cap G_{\varphi} \right)} \abs{Qu-\dbar{u}}^2 \, d \mcH^1 \leq 8 \left[ \int_{\Omega_\varphi\cap \BB}\abs{Qu-\dbar{u}}^2 \, dx + \int_{G_\varphi} \abs{Qu-\dbar{u}}^2 \, d\mcH^2 \right]
\end{equation}
and
\begin{equation} \label{ineq:hyb 2}
\int_{\Omega_\varphi \cap \partial B^3(0,\sigma)}\abs{\nabla_{tan}(Qu)}^2 \, d \mcH^2 \leq 8 \int_{\Omega_{\varphi}\cap \BB}\abs{\nabla(Qu)}^2 \, dx.
\end{equation}
By Lemma \ref{lemma: Poincare constant} in the Appendix, there exists a constant $c$ independent of $\varphi$ such that
\begin{equation*}
\int_{\Omega_{\varphi}\cap \BB}\abs{u-\dbar{u}}^2 \, dx \leq c \int_{\Omega_\varphi\cap \BB}\abs{\nabla u }^2 \, dx
\end{equation*}
\begin{equation}\label{ineq:hyb 3}
\int_{G_{\varphi}}\abs{u-\dbar{u}}^2 \, d\mcH^2 \leq c \int_{\Omega_\varphi\cap \BB}\abs{\nabla u }^2 \, dx.
\end{equation}

Combining inequalities \eqref{ineq:bound on nabla Qu}, \eqref{ineq:hyb 1},  \eqref{ineq:hyb 2}, \eqref{ineq:hyb 3} and Lemma \ref{lem: Qn-n bound} we have
\begin{align*}
&\left( \int_{\Omega_\varphi \cap \partial B^3(0,\sigma)} \abs{\nabla_{tan}(Qu)}^2 \, d\mcH^2 \right)\left(\int_{\partial\left( B^3(0,\sigma) \cap G_\varphi \right)}\abs{Qu - \dbar{u}}^2 \, d\mcH^1 + \int_{\Omega_\varphi \cap \partial B^3(0,\sigma)} \abs{Qu -\dbar{u}}^2 \, d \mcH^2 \right) \\
&\qquad \leq 64 \left( \int_{\Omega_\varphi\cap \BB}\abs{\nabla (Qu)}^2 \, dx \right)\left(\int_{ G_\varphi}\abs{Qu - \dbar{u}}^2 \, d\mcH^2 + \int_{\Omega_\varphi\cap \BB} \abs{Qu -\dbar{u}}^2 \, dx \right) \\
&\qquad\leq 2304 \left( \liptwo{\varphi}^2\mathcal{L}^3(\Omega_\varphi\cap \BB)+ \int_{\Omega_\varphi\cap \BB}\abs{\nabla u }^2 \, dx \right)\times \\ &\qquad \qquad \qquad \left( \int_{\Omega_\varphi\cap \BB} \abs{Qu -u}^2 +\abs{u - \dbar{u}}^2 \, dx + \int_{G_\varphi}\abs{Qu - u}^2 + \abs{u - \dbar{u}}^2 \, d \mcH^2 \right) \\
&\qquad\leq 2304 c \left( \liptwo{\varphi}^2+ \int_{\Omega_\varphi\cap \BB}\abs{\nabla u }^2 \, dx \right)\times \\ &\qquad \qquad \qquad \left( \lip{\varphi}\mcH^2(G_{\varphi}) + \lip{\varphi}\mathcal{L}^3(\Omega_\varphi) + 2c \int_{\Omega_{\varphi}\cap \BB}\abs{\nabla u}^2 \, dx \right) \\
&\qquad\leq  c \left( \liptwo{\varphi}^2+ \int_{\Omega_\varphi\cap \BB}\abs{\nabla u }^2 \, dx \right) \left( \lip{\varphi} +  \int_{\Omega_{\varphi}\cap \BB}\abs{\nabla u}^2 \, dx \right)
\end{align*}
where $c$ has absorbed all constants. Note that as
\begin{equation}\label{ineq: bounds on measure of Omega} 
 \mathcal{L}^3(\Omega_\varphi\cap \BB) \leq  \mathcal{L}^3(\BB) \text{ and } \mcH^2(G_{\varphi}) \leq \int_{\BB^2} 1 + \abs{\nabla \varphi }^2 dx \leq 2 \mathcal{L}^2(\BB^2)
 \end{equation}
the constant $c$ can be chosen to not depend on the domain. Defining $\varepsilon := a\lambda$, for some $0<a<1$ to be chosen later, and  choosing $c_1$ such that $c_1^2< a^{q}\frac{ \delta^2}{c}$, where $\delta>0$ is the constant from Lemma \ref{lem: exten lem}, we have
\begin{align*}
\left( \int_{\Omega_\varphi \cap \partial B^3(0,\sigma)} \abs{\nabla_{tan}(Qu)}^2 \, d\mcH^2 \right)\left(\int_{\partial\left( B^3(0,\sigma) \cap G_\varphi \right)}\abs{Qu - \dbar{u}}^2 \, d\mcH^1 + \int_{\Omega_\varphi \cap \partial B^3(0,\sigma)} \abs{Qu -\dbar{u}}^2 \, d \mcH^2 \right) \\ 
\leq  c_1^2 \lambda^q c \\
\leq  \varepsilon^q \delta^2.
\end{align*}
It now follows that $Qu$ satisfies the small oscillation condition of Lemma \ref{lem: exten lem}. Therefore, there exists $\omega \in W^{1,2}(\Omega_\varphi\cap B^3(0,\sigma),\SSS^2)$ such that $\left. \omega\right|_{\Omega_\varphi \cap \partial B^3(0,\sigma)} = \left. Qu\right|_{\Omega_\varphi \cap \partial B^3(0,\sigma)}$ , $\left. \omega\right|_{G_\varphi \cap  B^3(0,\sigma)} \in \tilde{\SSS}^1$ and
\begin{align*}
\int_{\Omega_\varphi \cap B^3(0,\sigma)} \abs{\nabla \omega}^2 \, dx &\leq c\varepsilon \int_{\Omega_\varphi \cap \partial B^3(0,\sigma)} \abs{\nabla_{tan}Qu}^2 \, d \mcH^2 \\ &\qquad + c\varepsilon^{-q} \left[ \int_{\Omega_\varphi \cap \partial B^3(0,\sigma)} \abs{Qu - \dbar{u}}^2 \, d \mcH^2 + \int_{\partial (\Omega_\varphi \cap \partial B^3(0,\sigma))} \abs{Qu- \dbar{u}}^2 \, d \mcH^1 \right]. 
\end{align*}
Using \eqref{ineq:bound on nabla Qu}, \eqref{ineq:hyb 1},  \eqref{ineq:hyb 2}, \eqref{ineq:hyb 3}, \eqref{ineq: bounds on measure of Omega} and Lemma \ref{lem: Qn-n bound} we bound
\begin{align}
\int_{\Omega_\varphi \cap B^3(0,\sigma)} \abs{\nabla \omega}^2 \, dx &\leq 8\varepsilon \int_{\Omega_\varphi\cap \BB}\abs{\nabla Qu}^2 \, dx + 8 c \varepsilon^{-q} \left[ \int_{\Omega_\varphi\cap \BB} \abs{Qu - \dbar{u}}^2 \, dx + \int_{G_\varphi}\abs{Qu- \dbar{u}}^2 \, d \mcH^2 \right] \notag \\
&\leq c \varepsilon \left( \liptwo{\varphi}^2 +  \int_{\Omega_\varphi\cap \BB}\abs{\nabla u }^2 \, dx\right) \notag  \\ &\qquad + c \varepsilon^{-q}\left[\lip{\varphi}^2 + \int_{\Omega_\varphi\cap \BB} \abs{u - \dbar{u}}^2  \, dx + \int_{G_\varphi} \abs{u-\dbar{u}}^2 \, d \mcH^2 \right]\label{ineq:hyb 4}
\end{align}
where $c$ has absorbed all constants.

We now use the inverse of the matrix $Q(x)$ in order to get a map that belongs to $\mathcal{U}_\varphi$. For a given point $x \in \partial \Omega_\varphi$ the matrix $[Q(x)]^{-1}$ is given by $Q(x)^T$, thus using the symmetry of $Q$ we can write $[Q(x)]^{-1}$ out explicitly as
\begin{equation*}
[Q(x)]^{-1}= \left[ 
\begin{matrix}
Q_{11}(x) & Q_{12}(x) & -Q_{13}(x) \\ Q_{21}(x) & Q_{22}(x) & -Q_{23}(x) \\ -Q_{31}(x) & -Q_{32}(x) & Q_{33}(x)
\end{matrix} \right],
\end{equation*}
where $Q_{ij}(x)$ are the entries of the matrix $Q(x)$. By the same calculations as before the map $x \mapsto [Q(x)]^{-1}$ is $C^1$ and satisfies the bound
\begin{equation}\label{ineq:hyb 5}
\abs{\nabla([Q(x)]^{-1}\omega(s))}^2 \leq 18\left( \liptwo{\varphi}^2 + c \abs{\nabla \omega}^2\right).
\end{equation}
Moreover, as $\left. \omega\right|_{G_\varphi \cap  B^3(0,\sigma)} \in \tilde{\SSS}^1$ we have $\tilde{w}(x) : = [Q(x)]^{-1}\omega(x) \in \mcU_{\Omega_\varphi \cap B^3(0,\sigma)}$ and $\left. \tilde{w}(x) \right|_{\Omega_\varphi \cap \partial B^3(0,\sigma)} = \left. u\right|_{\Omega_\varphi \cap \partial B^3(0,\sigma)}$.

Using that $\sigma>\frac{1}{2}$ and the minimality of $u$ we have
\begin{align*}
&\int_{\Omega_{\varphi}\cap B^3(0,\frac{1}{2})} |\nabla u|^2 \,dx - K_{13}\int_{G_\varphi \cap B^3(0,\frac{1}{2})} u^\alpha u^\beta  \frac{\partial \nu^\alpha }{\partial x^\beta}\,d\mcH^2 \\
&\quad \leq \int_{\Omega_{\varphi}\cap B^3(0,\sigma)} |\nabla u|^2 \,dx - K_{13}\int_{G_\varphi \cap B^3(0,\sigma)} u^\alpha u^\beta  \frac{\partial \nu^\alpha }{\partial x^\beta}\,d\mcH^2 \\ &\qquad + K_{13} \left[\int_{G_\varphi \cap B^3(0,\sigma)} u^\alpha u^\beta  \frac{\partial \nu^\alpha }{\partial x^\beta}\,d\mcH^2  - \int_{G_\varphi \cap B^3(0,\frac{1}{2})} u^\alpha u^\beta  \frac{\partial \nu^\alpha }{\partial x^\beta}\,d\mcH^2 \right] \\
&\quad \leq \int_{\Omega_{\varphi}\cap B^3(0,\sigma)} |\nabla \tilde{w}|^2 \,dx \\ &\qquad + K_{13} \left[\int_{G_\varphi \cap B^3(0,\sigma)} \left(u^\alpha u^\beta - \tilde{w}^\alpha \tilde{w}^\beta\right)  \frac{\partial \nu^\alpha }{\partial x^\beta}\,d\mcH^2  - \int_{G_\varphi \cap B^3(0,\frac{1}{2})} u^\alpha u^\beta  \frac{\partial \nu^\alpha }{\partial x^\beta}\,d\mcH^2 \right] 
\end{align*}
which implies the inequality
\begin{align*}
&\int_{\Omega_{\varphi}\cap B^3(0,\frac{1}{2})} |\nabla u|^2 \,dx \leq \int_{\Omega_{\varphi}\cap B^3(0,\sigma)} |\nabla \tilde{w}|^2 \,dx + K_{13} \left[\int_{G_\varphi \cap B^3(0,\sigma)} \left(u^\alpha u^\beta - \tilde{w}^\alpha \tilde{w}^\beta\right)  \frac{\partial \nu^\alpha }{\partial x^\beta}\,d\mcH^2   \right] .
\end{align*}
Next we apply \eqref{ineq:hyb 5} and \eqref{ineq:hyb 4}, choose $a < \frac{1}{18 c}$ and substitute $\varepsilon= a\lambda $ to get the bound
\begin{align*}
\int_{\Omega_{\varphi}\cap B^3(0,\sigma)} |\nabla \tilde{w}|^2 \,dx &\leq \int_{\Omega_{\varphi}\cap B^3(0,\sigma)} 18\left( \liptwo{\varphi}^2 + c \abs{\nabla \omega}^2\right)\,dx   \\
&\leq 18 \Bigg( c  \varepsilon \left( \liptwo{\varphi}^2 +  \int_{\Omega_\varphi\cap \BB}\abs{\nabla u }^2 \, dx\right) \\ &\quad + c \varepsilon^{-q}\left[\lip{\varphi}^2 + \int_{\Omega_\varphi} \abs{u - \dbar{u}}^2 \, dx + \int_{G_\varphi} \abs{u-\dbar{u}}^2 \, d \mcH^2 \right]   \Bigg) \notag \\ &\quad + 18\liptwo{\varphi}^2\mathcal{L}^3(\Omega_\varphi)  \\
&\leq  \lambda \left( \liptwo{\varphi}^2 + \int_{\Omega_\varphi\cap \BB} \abs{\nabla u}^2 \, dx \right) \\ &\quad + \lambda^{-q}c_2 \left( \lip{\varphi}^2 + \int_{\Omega_\varphi\cap \BB} \abs{u - \dbar{u}}^2 \, dx + \int_{G_\varphi} \abs{u - \dbar{u}}^2 d\mcH^2 \right) \notag \\ &\quad+  c_2 \liptwo{\varphi}^2.
\end{align*}
On the other hand, as $u$ and $\tilde{w}$ are both $\SSS^2$ valued, we have 
\begin{align*}
K_{13} \left[\int_{G_\varphi \cap B^3(0,\sigma)} \left(u^\alpha u^\beta - \tilde{w}^\alpha \tilde{w}^\beta\right)  \frac{\partial \nu^\alpha }{\partial x^\beta}\,d\mcH^2   \right] &\leq |K_{13}| \int_{G_\varphi \cap B^3(0,\sigma)} \left|u^\alpha u^\beta - \tilde{w}^\alpha \tilde{w}^\beta\right|  \left|\frac{\partial \nu^\alpha }{\partial x^\beta} \right| \,d\mcH^2   \\
&\leq c \int_{G_\varphi \cap B^3(0,\sigma)} \left|\frac{\partial \nu^\alpha }{\partial x^\beta} \right| \,d\mcH^2.
\end{align*}

A straight forward calculation gives
\begin{equation*}
\frac{\partial \nu}{\partial x^\beta} = \left[ \frac{\left(-\varphi_{x^1,x^\beta},-\varphi_{x^2,x^\beta},0 \right)}{\left(\varphi_{x^1}^2 +\varphi_{x^2}^2 + 1 \right)^{\frac{1}{2}}} - \frac{\left(\varphi_{x^1,x^\beta}\varphi_{x^1} + \varphi_{x^2,x^\beta}\varphi_{x^2}\right) \left( -\varphi_{x^1} , -\varphi_{x^2}, 1  \right) }{\left(\varphi_{x^1}^2 +\varphi_{x^2}^2 + 1 \right)^{\frac{3}{2}}} \right]
\end{equation*}
hence 
\begin{equation*}
\abs{\frac{\partial \nu^\alpha}{\partial x^\beta} }\leq 3 \liptwo{\varphi}.
\end{equation*}
Recalling that $\mcH^2(G_\varphi)$ is bounded for all $\varphi$ we complete the proof.

\end{proof}

\begin{boxlemma}{\bf [Energy Improvement]}\label{lem:Energy Improvement}
There are positive constants $\varepsilon,c$ and $\theta <1$ such that if $\varphi$ is as in \eqref{def:varphi}, $u$ is a minimizer of $\mathbb{G}_\varphi$ amongst maps in $\mcU_{\varphi}$ with fixed trace on $H_{\varphi}$ and $\int_{\Omega_\varphi}\abs{\nabla u}^2 dx \leq \varepsilon^2$, then
\begin{equation}\label{inequ: Energy Improvement}
\frac{1}{\theta} \int_{C(0,\theta) \cap \Omega_\varphi} \abs{ \nabla u}^2 dx \leq \theta \max \{ \int_{\Omega_\varphi} \abs{ \nabla u}^2 dx , c \left( \lip{\varphi} + \liptwo{\varphi}\right)\}
\end{equation}
\end{boxlemma}
\begin{proof}
Suppose, for a contradiction, that for fixed $0<\theta<1$ there are sequences $(u_i),(\varphi_i),(\varepsilon_i)$ such that $u_i$ is minimizing in $\Omega_{\varphi_i}$ and $\varepsilon_i^2 = \int_{\Omega_{\varphi_i}}\abs{\nabla u_i}^2 dx \rightarrow 0$ as $i \rightarrow \infty$, but
$$\frac{1}{\theta}\int_{C(0,\theta) \cap \Omega_{\varphi_i}}\abs{\nabla u_i}^2dx > \theta \max \left\{ \int_{\Omega_{\varphi_i}} \abs{\nabla u_i}^2 dx , i \left(\lip{\varphi_i} + \liptwo{\varphi_i}\right)\right\}.$$
This implies that 
\begin{align}
\frac{1}{\theta} \int_{C(0,\theta) \cap \Omega_{\varphi_i}} \abs{\nabla u_i}^2 dx &> \theta\varepsilon_i^2, \\
\frac{\lip{\varphi_i}}{\varepsilon_i^2}&\rightarrow 0,\label{ineq: lip phi_i bound} \\ 
\frac{\liptwo{\varphi_i}}{\varepsilon_i^2} &\rightarrow 0. \label{ineq: lip2 phi_i bound} 
\end{align}
We now consider the normalised functions
\begin{equation*}
v_i = \varepsilon_i^{-1}({u}_i - \dbar{u}_i).
\end{equation*}
Then, by Lemma \ref{lemma: Poincare constant}, the sequence $\left\{ \norm{v_i}_{W^{1,2}(\Omega_{\varphi_i})}\right\}_{i=1}^\infty$ is a bounded sequence in $\RR$. As each $v_i$ is defined on a different domain we extend them all to a common domain using the extension Lemma \ref{lemma:extension lemma} that can be found in the appendix. Let $\mathcal{C}=\{x\in \RR^3 : \abs{(x^1,x^2)}<1 \text{ and } -1 < x^3 < \frac{1}{2}\}$. Then there exists, by Lemma \ref{lemma:extension lemma}, a constant $c$, that is independent of $i$, and functions $\widehat{v}_i \in W^{1,2}(\mathcal{C})$ such that $v_i = \left. \widehat{v}_i\right|_{\Omega_{\varphi_i}}$ and 
\begin{equation*}
\norm{\widehat{v}_i}_{W^{1,2}(\mathcal{C})} \leq c \norm{v_i}_{W^{1,2}(\Omega_{\varphi_i})}
\end{equation*}
(note that $\Omega_{\varphi_i} \subset \mathcal{C}$ for $i$ sufficiently large and so we assume $\Omega_{\varphi_i} \subset \mathcal{C}$ for all $i$). As $\widehat{v_i}$ is bounded in $W^{1,2}(\mathcal{C})$ there exists $\widehat{v} \in W^{1,2}(\mathcal{C})$ such that $\widehat{v_i}$ converges weakly (on a subsequence) to $\widehat{v}$ in $W^{1,2}(\mathcal{C})$. Define $v \in W^{1,2}(\Omega_0)$ as $v:= \left.\widehat{v}\right|_{\Omega_0}$.

We claim that the function $v$ is harmonic. In order to see this we first note that as $u_i$ are minimizers and $u_i(x)\cdot \nu(x) = 0$ for $x \in \partial G_{\varphi_i}$ we have that $u_i$ satisfies
\begin{equation}\label{eq:weak harmonic maps equation}
\int_{\Omega_{\varphi_i}}\langle \nabla u_i , \nabla \zeta\rangle - \abs{ \nabla u_i}^2 u_i \cdot \zeta dx= 0,
\end{equation}
for all $\zeta \in C^\infty_0(\Omega_{\varphi_i})$. Let $\zeta \in C^\infty(\RR^3)$ be such that $spt(\zeta) \subset \Omega_0$. Then for sufficiently large $i$ we have $spt(\zeta) \subset \Omega_{\varphi_i}$. Observing that $\nabla u_i = \varepsilon_i \nabla v_i$, substituting this into \eqref{eq:weak harmonic maps equation} and dividing by $\varepsilon_i$ yields
\begin{align*}
0 = \int_{\Omega_{\varphi_i}} \langle \nabla v_i , \nabla \zeta \rangle - \varepsilon_i \abs{\nabla v_i}^2 u_i \cdot \zeta dx.
\end{align*}
Using the uniform bounds on $\abs{\nabla v_i}^2$,$u_i$ and $\zeta$, we see that the second term tends to $0$ as $i \rightarrow \infty$. As $spt(\zeta) \subset \Omega_0 \subset \mathcal{C}$ and we have that $spt(\zeta) \subset \Omega_{\varphi_i} \subset \mathcal{C}$ for $i$ large enough, we can use the weak convergence of $\widehat{v}_i$ to $\widehat{v}$ to get

\begin{align*}
0 &=\lim_{i\rightarrow \infty}  \int_{\Omega_{\varphi_i}} \langle \nabla v_i , \nabla \zeta \rangle - \varepsilon_i \abs{\nabla v_i}^2 u_i \cdot \zeta dx \\
&= \lim_{i \rightarrow \infty}\left[\int_{\mathcal{C}} \langle \nabla v_i ,\nabla \zeta \rangle dx\right] + 0 \\
&= \int_{\mathcal{C}}\langle \nabla v ,\nabla \zeta \rangle dx \\
&= \int_{\Omega_0} \langle \nabla v ,\nabla \zeta \rangle dx.
\end{align*}
Therefore $v$ is harmonic in $\Omega_0$.

 We now examine the behaviour of $v$ on $G_0$ and we will show that $v$ is regular up to the boundary. In order to do so we introduce the following subspaces of $\SSS^2$:
 \begin{align*}
\Sigma_x^i :&= \{y \in \SSS^2: y \cdot \nu(x^1,x^2,\varphi_i(x^1,x^2)) = 0\} \\
\Sigma_0 :&= \{y \in \SSS^2:y \cdot (0,0,1) =0 \}.
\end{align*}

\noindent \emph{Claim: $v(x) \in {\rm Tan}(\Sigma_0,a)$ for $\mcH^2$ almost everywhere $x \in G_0$.}
\bigskip

\noindent \emph{Proof of Claim:}
For each $i$ let $Q_i$ be the projection defined by \eqref{def: Q def } (that now depends on $i$). Let $\tilde{u}_i:\Omega_0 \rightarrow \SSS^2$ be the functions defined by 
\begin{equation*}
\tilde{u}_i(x):= u_i(\tilde{x}_i), \quad \text{ where } \tilde{x}_i:= (x^1,x^2,x^3 + \varphi_i(x^1,x^2)). 
\end{equation*}
and $\tilde{v}_i$ to be the functions 
\begin{equation*}
\tilde{v}_i = \varepsilon_i^{-1}(\tilde{u}_i - \dbar{u}_i).
\end{equation*}
Note that for almost every $x \in \Omega_0$ we have $\tilde{v}_i(x) \rightarrow v(x)$.
As $u_i(x)\in \Sigma_x^i$ for almost every $x\in G_{\varphi_i}$, we have $Q_i(\tilde{x}_i)\cdot \tilde{u}_i(x)\in \Sigma_0$ for almost every $x \in G_0$. Then for almost every $x \in G_0$ we have
\begin{equation*}
{\rm dist}(\dbar{u}_i,\Sigma_0)^2 \leq \abs{Q_i(\tilde{x}_i)\tilde{u}_i(x) - \dbar{u}_i}^2 .
\end{equation*}
Averaging this over $G_0$ and using Lemma \ref{lem: Qn-n bound}, relation \eqref{ineq: lip phi_i bound} and the Poincar\'{e} inequalities gives
\begin{align}
{\rm dist}(\dbar{u}_i,\Sigma_0)^2 &\leq \left[\mcH^2(G_0)\right]^{-1} \int_{G_0}\abs{Q_i(x)\tilde{u_i} - \dbar{u}_i}^2d\mcH^2 \notag \\
&\leq 2\left[\mcH^2(G_0)\right]^{-1} \left( \int_{G_0}\abs{Q_i(\tilde{x}_i)\tilde{u_i}(x) - \tilde{u_i}(x)}^2 + \abs{\tilde{u_i} - \dbar{u}_i}^2 \mcH^2 \right) \notag \\
&\leq c \varepsilon_i^2. \label{ineq:bound on dist}
\end{align}  
Hence, for $i$ sufficiently large, there is a unique nearest point $a_i$ of $\dbar{u}_i$ on $\Sigma_0$. As $(\dbar{u}_i)_{i \in \NN}$ is a bounded sequence in $\RR^3$ it has a subsequence converging to some $a \in \RR^3$. Also $\varepsilon_i^{-1}\abs{\dbar{u}_i-a_i} = \frac{{\rm dist}(\dbar{u}_i,\Sigma_0)}{\varepsilon_i}\leq \frac{\sqrt{c} \varepsilon_i}{\varepsilon_i} = \sqrt{c}$, is bounded. Thus on subsequence we have
\begin{equation*}
\lim_{i \rightarrow \infty} \dbar{u}_i = a \in \Sigma_0 \text{ and } \lim_{i \rightarrow} \varepsilon_i^{-1}(\dbar{u}_i - a_i) = w \in \RR^3.
\end{equation*}
As $(\dbar{u}_i-a_i) \in {\rm Nor}(\Sigma_0,a_i)$ and $a_i \rightarrow a$ we have $w \in {\rm Nor}(\Sigma_0,a)$.
For almost every $x \in G_0$, $\tilde{v_i}(x) \rightarrow v(x)$ as $i\rightarrow \infty$. For such an $x$ and $i$ sufficiently large we have that
 $$(x^1,x^2,x^3+ \varphi_i(x^1,x^2)) \in G_{\varphi_i},$$
and hence $\tilde{u}_i(x) \in \Sigma^i_{\tilde{x}_i}$. We compute
\begin{align*}
\lim_{i \rightarrow \infty} \varepsilon_i^{-1}(Q_i(x)\tilde{u}_i(x) - \dbar{u}_i) &= \lim_{i \rightarrow \infty}\left(\varepsilon_i^{-1}(Q_i(x)\tilde{u}_i(x)- \tilde{u}_i(x)) + \varepsilon_i^{-1}(\tilde{u}_i(x) - \dbar{u}_i)\right) \\
&= \lim_{i \rightarrow \infty} \varepsilon_i^{-1}\left(Q_i(x)\tilde{u}_i(x) - \tilde{u_i}(x) \right) + v(x),
\end{align*}
using Lemma \ref{lem: Qn-n bound} 
\begin{align*}
\lim_{i \rightarrow \infty} \varepsilon_i^{-1}\abs{Q_i(x)\tilde{u_i}(x) - \tilde{u_i}(x)} &\leq \lim_{i\rightarrow\infty} c\frac{\lip{\varphi_i}}{\varepsilon_i} \\
&=0.
\end{align*}
Thus
\begin{equation*}
\lim_{i \rightarrow \infty} \varepsilon_i^{-1}\left( Q_i(x)\tilde{u_i}(x) - \dbar{u}_i\right) = v(x) \text{ for almost every }x \in G_0.
\end{equation*}
As $\varepsilon_i^{-1}(Q_i(x)\tilde{u}_i(x) - a_i)$ approaches a vector in ${\rm Tan}(\Sigma_0,a)$ and $w \in {\rm Nor}(\Sigma_0,a)$, we have
$$\varepsilon_i^{-1}(\tilde{u_i}(x) - a_i) \cdot \varepsilon_i^{-1}(\dbar{u}_i - a_i) \rightarrow 0 \text{ as } i \rightarrow \infty.$$
Thus
\begin{align*}
v(x)\cdot(-w) &= \lim_{i \rightarrow}\varepsilon_i^{-1}(\tilde{u_i}(x) - \dbar{u}_i) \cdot \varepsilon_i^{-1}(a_i - \dbar{u}_i) \\
&= \lim_{i \rightarrow} \left( \varepsilon_i^{-1}(\tilde{u_i}(x) - a_i) \varepsilon_i^{-1}(a_i - \dbar{u}_i)  \right)\cdot \varepsilon_i^{-1}(a_i - \dbar{u}_i) \\
&= \abs{w}^2.
\end{align*}
By averaging over $G_0$ we deduce $\abs{w}^2 = \dbar{v}(-w) = 0$, hence
\begin{align*}
v(x) &= v(x) + w \\
&= \lim_{i \rightarrow}\varepsilon_i^{-1}(Q_i(x)\tilde{u_i}(x) - \dbar{u}_i) + \lim_{i \rightarrow}\varepsilon_i^{-1}(\dbar{u}_i - a_i) \\
&= \lim_{i \rightarrow}\varepsilon_i^{-1}(Q_i(x)\tilde{u_i}(x) - a_i) \in{\rm Tan}(\Sigma_0 ,a).
\end{align*}   
This proves the claim.

Next we decompose $v = v^{\top} + v^{\bot}$, where $v^{\top} \in {\rm Tan}(\Sigma_0 ,a )$ and $v^{\bot} \in{\rm Nor}(\Sigma_0,a)$. We deduce that both $v^{\top}$ and $v^{\bot}$ are Harmonic inside $\Omega_0$ and that $v^{\bot}$ is regular up to $G_0$ because it satisfies the boundary condition 
$$v^{\bot} = 0 \text{ on } G_0.$$
To verify the regularity of $v^{\top}$ up to $G_0$ we show $v^{\top}$ satisfies the Neumann boundary condition 
$$\frac{\partial}{\partial x_m}v^{\top} = 0 \text{ on }G_0$$
in a weak sense, i.e
$$ \int_{\Omega_0} \nabla v \cdot \nabla \xi \, dx = 0$$ 
for any $\xi \in C^{\infty}(\overline{\Omega_0} , {\rm Tan}(\Sigma_0 , a))$ with $(\partial \Omega_0 \setminus G_0 )\cap {\rm Spt}(\xi) = \emptyset$. For this purpose choose an open neighbourhood $U$ of $\Sigma_0$ in $\RR^3$ such that every point $y \in U$ has a unique nearest point on  $\Sigma_0$. For $x \in G_{\varphi_i}$ and  $y \in U$, define the 1-dimensional subspaces of $\RR^3$ as
$$T_i(x,y) =  \{t(y \times \nu_i(x)) + y : t \in \RR\}.$$
Then for $x=(x^1,x^2,x^3) \in \Omega_{\varphi_i}$ and  $y \in U$, define 
\begin{equation*}
T_i(x,y) := T_i\big((x^1,x^2,\varphi_i(x^1,x^2)),y\big) .
\end{equation*}
We then have that $\{T_i(x,y):x \in \Omega_{\varphi_i},y \in U\}$ is a smooth field of 1 dimensional subspaces such that 
\begin{align*}
T_i(x,y) &\subseteq {\rm Tan}(\SSS^2 , y) \text{ for } x \in G_{\varphi_i}, y \in \SSS^2 \cap U 
\end{align*}	
and
\begin{align*}
T_i(x,y) &= {\rm Tan}(\Sigma^i_x,y) \text{ for } x \in G_{\varphi_i} \text{ and } y \in \Sigma_x^i.
\end{align*}
Next, define $\Pi_i : \Omega_{\varphi_i} \times U \times \RR^3 \rightarrow \RR^3$ such that for $(x,y,z) \in \Omega_{\varphi_i} \times U \times \RR^3$, $\Pi_i(x,y,z)$ is the orthogonal projection of $z$ onto $T_i(x,y)$. Explicitly $\Pi_i$ is given by
\begin{equation} \label{def:proj onto tangent space}
\Pi_i(x,y,z) = \frac{ \left[ (y \times \nu_i(x)) \otimes (y \times \nu_i(x))\right]z}{\abs{y \times \nu_i(x)}^2} + y.
\end{equation}
We have the following bounds on the derivatives of $\Pi_i$:
\begin{align}\label{inequal: bounds on derivatives of Pi}
\abs{\frac{\partial \Pi_i(x,y,z)}{\partial x^j}} < C \varepsilon_i^2\abs{z}, \qquad
\abs{\frac{\partial \Pi_i(x,y,z)}{\partial z^j}} \leq 1, \qquad
\abs{\frac{\partial \Pi_i(x,y,z)}{\partial y^j}} < C\abs{z},
\end{align} 
the proof these bounds can be found in Lemma \ref{lemma:appendix bounds on derivatives of Pi} of the appendix.

We are now in a position to show that 
\begin{equation*}
\int_{\Omega_0}\nabla v \cdot \nabla \xi dx = 0.
\end{equation*}
We use the cut off function 
\begin{align*}
\lambda_\delta(t) := \begin{cases}
1 &\quad \text{ if } 0 \leq t \leq \frac{\delta}{2}, \\
2-2\delta^{-1} t &\quad \text{ if } \frac{\delta}{2} < t \leq \delta, \\
0 &\quad \text{ if } \delta < t. 
\end{cases}
\end{align*}
to define
$$\xi_i(x) = \lambda_{\delta_i}({\rm dist}(u_i(x),\Sigma_0)) \cdot \Pi_i(x,u_i(x),\xi(x)),$$
where $\delta_i$ will be determined later. We have
\begin{equation*}
\begin{cases}
	\xi_i(x) \in {\rm Tan}(\SSS^2 , u_i(x)) \text{ for } x \in \Omega_{\varphi_i}, \\
	\xi_i(x) \in {\rm Tan}(\Sigma_x^i,u_i(x)) \text{ for } x \in G_{\varphi_i}.
\end{cases}
\end{equation*}
Next, for $x \in \overline{\Omega}_{\varphi_i}$ let $u_i^t(x)$ be the solution of 
\begin{equation*}
\begin{cases}
	\left(\frac{d}{dt}\right)u_i^t(x) = \xi_i(x), \\
	u_i^0(x) = u_i(x).
\end{cases}
\end{equation*}
Then $u_i^t \in W^{1,2}(\Omega_{\varphi_i},\SSS^2)$ with $\frac{d}{dt}u_i^t \in {\rm Tan}(\Sigma_x^i,u_i(x))$ for $x \in G_{\varphi_i}$, hence $u_i^t(x) \in \Sigma_x^i$ for $x \in G_{\varphi_i}$ and $u_i^t(x) = u_i(x)$ for $x \in \partial \Omega_{\varphi_i} \setminus G_{\varphi_i}$. The minimality of $u_i$ implies that 
\begin{align*}
0 &= \left.\frac{d}{dt}\right|_{t=0} \int_{\Omega_{\varphi_i}}\abs{\nabla u_i^t(x)}^2 \, dx - K_{13} \int_{G_{\varphi_i}} u^{t,\alpha}_i u^{t,\beta}_i \frac{\partial \nu^\beta}{\partial x^\alpha} \, d\mcH^2 \\ 
&= 2 \int_{\Omega_{\varphi_i}}\nabla u_i \nabla \xi_i \, dx - K_{13}\int_{G_{\varphi_i}} \xi_i^\alpha u_i^\beta \frac{\partial \nu^\beta}{\partial x^\alpha} + \xi_i^\beta u_i^\alpha \frac{\partial \nu^\beta}{\partial x^\alpha}\, d\mcH^2 .
\end{align*}
Let $A_i = \{x \in \overline{\Omega}_{\varphi_i} : \abs{u_i(x) - \dbar{u}_i}^2 \geq \frac{\delta_i^2}{4}\}$. We have
$$\int_{A_i} \frac{\delta_i^2}{4} \, dx \leq \int_{A_i} \abs{u_i(x) - \dbar{u}_i}^2 \, dx \leq \int_{\Omega_{\psi_i}} \abs{u_i(x) - \dbar{u}_i}^2 \, dx \leq
 c \int_{\Omega_{\psi_i}} \abs{\nabla u_i}^2 \, dx \leq c \varepsilon_i^2$$
and hence
\begin{equation}\label{inequal: Measure of Ai}
\mathcal{L}^3(A_i) \leq  4c\varepsilon_i^2 \delta_i^{-2}.
\end{equation}
Setting $B_i = \Omega_{\varphi_i}\setminus A_i$ we have
\begin{align*}
\abs{\int_{\Omega_{\varphi_i}}\nabla v_i \cdot \nabla \xi \, dx } &= \varepsilon_i^{-1} \Bigg| \int_{\Omega_{\varphi_i}}\nabla u_i \cdot \nabla \xi - \nabla u_i \cdot \nabla \xi_i \, dx   \\ &\qquad \qquad +   \frac{K_{13}}{2} \int_{G_{\varphi_i}} \xi_i^\alpha u_i^\beta \frac{\partial \nu^\beta}{\partial x^\alpha} + \xi_i^\beta u_i^\alpha \frac{\partial \nu^\beta}{\partial x^\alpha}\, d\mcH^2 \Bigg| \\
&\leq   \underbrace{\varepsilon_i^{-1} \abs{ \int_{A_i}\nabla u_i \cdot\left( \nabla \xi - \nabla \xi_i\right)\, dx }}_{:=I_i} +\underbrace{\varepsilon_i^{-1} \abs{ \int_{B_i}\nabla u_i \cdot\left( \nabla \xi - \nabla \xi_i\right)\, dx }}_{:=II_i}  \\
&\qquad + \varepsilon_i^{-1} \Bigg|\frac{K_{13}}{2} \int_{G_{\varphi_i}} \xi_i^\alpha u_i^\beta \frac{\partial \nu^\beta}{\partial x^\alpha} + \xi_i^\beta u_i^\alpha \frac{\partial \nu^\beta}{\partial x^\alpha}\, d\mcH^2 \Bigg| \\
&\leq  I_i + II_i + c\varepsilon_i^{-1} \mcH^2(G_{\varphi_i})\norm{\varphi_i}_{C^2.} 
\end{align*}
We have that $ \varepsilon_i^{-1} c\mcH^2(G_{\varphi_i})\norm{\varphi_i}_{C^2}  \rightarrow 0$ as $i\rightarrow \infty$ and so we must show that $I_i,II_i \rightarrow 0 $ as $i\rightarrow 0$.\\

\bigskip
\noindent \framebox{$I_i\rightarrow 0$}\\
\bigskip

Using H\"{o}lder we have
\begin{align}
I_i &\leq \varepsilon_i^{-1} \left( \int_{A_i} \abs{\nabla u_i}^2 \, dx \right)^{\frac{1}{2}} \left( 2\left( \int_{A_i} \abs{\nabla \xi}^2 \,dx + \int_{A_i} \abs{\nabla \xi_i}^2 \,dx \right)\right)^{\frac{1}{2}} \notag \\
&\leq \varepsilon_i^{-1}(\varepsilon_i^2)^{\frac{1}{2}}\left(  2 (c \varepsilon_i^2 \delta_i^{-2}) + 2 \int_{A_i} \abs{\nabla \xi_i}^2 \, dx \right) \qquad \text{( as $ \abs{\nabla \xi(x)} < c$ for some $c$)} \label{ineq: bound on I_i}. \\
\end{align}
We now focus on $\abs{\nabla \xi_i}^2$. Writing 
\begin{equation*}
\Pi_i(x,y,z) = \left( \Pi_i^1(x,y,z) , \Pi_i^2(x,y,z) , \Pi_i^3(x,y,z)\right),\quad u_i = (u_i^1,u_i^2,u_i^3)\quad \text{ and } \xi = (\xi^1,\xi^2,\xi^3),
\end{equation*}
 we have 
\begin{align*}
\abs{\frac{\partial \xi^k_i}{\partial x^j}}^2 &= \abs{ \frac{\partial}{\partial x^j}\left( \lambda_{\delta_i}({\rm dist}(u_i(x),\Sigma_0)) \Pi_i^k(x,u_i(x),\xi(x))\right) }^2 \\
&= \left| \lambda'_{\delta_i}({\rm dist}(u_i(x),\Sigma_0)) \left( \sum_{l=1}^3 \frac{\partial u_i^l(x)}{\partial x^j} \frac{\partial({\rm dist} (u_i(x),\Sigma_0))}{\partial x^l} \right)\Pi_i^k(x,u_i(x),\xi(x)) \right. \\ &\qquad \qquad \left. + \lambda_{\delta_i}({\rm dist } (u_i(x),\Sigma_0)) \frac{\partial }{\partial x^j}\left( \Pi_i^k(x,u_i(x),\xi(x)) \right) \right|^2 \\
&\leq c \delta_i^{-1} \abs{\sum_{l=1}^3 \frac{\partial u_i^l(x)}{\partial x^j}}^2  + 2 \left| \sum_{l=1}^3 \left[ \frac{\partial \Pi_i^k(x,u_i(x),\xi(x)}{\partial x^j}\delta_{lj} \right. \right.\\ &\qquad \qquad \left. \left.  + \frac{\partial \Pi_i^k(x,u_i(x),\xi(x))}{\partial y^l}\frac{\partial u_i^l(x)}{\partial x^j} + \frac{\partial \Pi_i^k(x,u_i(x),\xi(x)}{\partial z^l}\frac{\partial \xi^l(x)}{\partial x^j} \right] \right|^2 \\
&\leq c \delta_i^{-1} \sum_{l=1}^3 \abs{\frac{\partial u_i^l(x)}{\partial x^j}}^2 + c \sum_{l=1}^3 \abs{\frac{\partial \Pi_i^k}{\partial x^l}(x,u_i(x),\xi(x))}^2 + c \sum_{1=l}^3 \abs{ \frac{\partial u_i^l}{\partial x^j}(x)}^2 + c\sum_{l=1}^3 \abs{ \frac{\partial \xi^l(x)}{\partial x^j}}^2,
\end{align*}
where in the last line we have used \eqref{inequal: bounds on derivatives of Pi}. Summing this inequality over $j,k$, using the estimates \eqref{inequal: bounds on derivatives of Pi}, \eqref{inequal: Measure of Ai} and the fact that $\int_{\Omega_{\varphi_i}}\abs{\nabla u}^2 dx = \varepsilon_i^2$ yields
\begin{align}
\int_{A_i}\abs{\nabla \xi_i}^2 dx &= \int_{A_i} \sum_{i,j=1}^3 \abs{\frac{\partial \xi_i^k(x)}{\partial x^j}}^2 dx \notag \\
&\leq \int_{A_i}\sum_{i,j=1}^3 \left[ c \delta_i^{-1} \sum_{l=1}^3 \abs{\frac{\partial u_i^l}{\partial x^j}}^2 + c \sum_{l=1}^3 \abs{\frac{\partial \Pi_i^k(x,u_i(x),\xi(x))}{\partial x_l}}^2 \right. \notag \\ &\qquad \qquad \left. +  c \sum_{l=1}^3 \abs{ \frac{\partial u_i^l(x)}{\partial x^j}}^2  + c \sum_{l=1}^3 \abs{\frac{\partial \xi^l(x)}{\partial x^j}}^2 \right] dx \notag \\
&\leq c \delta_i^{-1} \varepsilon_i^2 + c \varepsilon_i^2 + c\varepsilon_i^2 + c \varepsilon_i^2\delta_i^{-2}    \notag \\
&\leq c \delta_i^{-1}\varepsilon_i^2 + 4 c \varepsilon_i^2 \delta_i^{-2}(c + c) + c \varepsilon_i^2. \label{ineq: bound on nabla xi_i}
\end{align}
Therefore, by combining \eqref{ineq: bound on I_i} and \eqref{ineq: bound on nabla xi_i} we get
\begin{equation*}
I_i \leq  2(c\varepsilon_i^2 \delta_i^{-1}) + c \delta_i^{-1} \varepsilon^2_i + 4c \varepsilon_i^2\delta_i^{-2}(c\varepsilon_i^2 + c) + c \varepsilon_i^2  \rightarrow 0 \text{ as } i\rightarrow \infty,
\end{equation*}
providing $\delta_i=\varepsilon_i^{\frac{1}{3}}b$, where $b$ is a constant to be chosen later.

We now show $II_i \rightarrow 0$.
Recall that for $x \in B_i$ we have $\abs{u_i(x) - \dbar{u}_i}\leq \frac{\delta_i}{4}$ and so for $x \in B_i$
\begin{align*}
{\rm dist}(u_i(x) ,\Sigma_0) &\leq {\rm dist}(u_i(x),\dbar{u}_i) + {\rm dist}(\dbar{u_i},\Sigma_0) \\
&\leq \frac{\delta_i}{4} + c \varepsilon_i\qquad ( \text{by inequality \ref{ineq:bound on dist}}) \\
&\leq \frac{\delta_i}{2} \qquad (\text{ providing } \delta_i = 4c\varepsilon_i^{\frac{1}{3}}).
\end{align*}
For $x \in B_i$ we have 
\begin{align*}
\xi_i(x) &= \lambda_{\delta_i}({\rm dist}(u_i(x),\Sigma_0))\cdot \Pi_i(x,u_i(x),\xi(x)) \\
&= \Pi_i(x,u_i(x),\xi(x)).
\end{align*}
As $\xi \in {\rm Tan}(\Sigma_0,a)$ we have $\xi(x) = \Pi_i(0,a,\xi(x))$, hence
\begin{align*}
\varepsilon_i^{-1}\abs{\int_{B_i} \nabla u_i \left( \nabla \xi -\nabla \xi_i\right) \, dx } &= \varepsilon_i^{-1}\abs{\int_{B_i} \nabla u_i \left( \nabla \Pi_i(0,a,\xi)  -\nabla \Pi_i(x,u_i(x),\xi)\right)\, dx } \\
&\leq \underbrace{\varepsilon_i^{-1}\abs{\int_{B_i} \nabla u_i \left( \nabla \Pi_i(0,a,\xi)  -\nabla \Pi_i(0,\dbar{u}_i,\xi)\right) dx }}_{:=III_i} \\ &\qquad + \underbrace{\varepsilon_i^{-1}\abs{\int_{B_i} \nabla u_i \left( \nabla \Pi_i(x,u_i(x),\xi)  -\nabla \Pi_i(0,\dbar{u}_i,\xi)\right)\, dx }}_{:=IV_i}. \\
\end{align*}
We now show that $III_i$ and $IV_i$ both go to zero separately. 
Using H\"{o}lder's inequality and Lemma \ref{lem:diff between derivatives of Pi} from the appendix we estimate 
\begin{align*}
III_i &\leq \varepsilon_i^{-1} \left( \int_{B_i} \abs{ \nabla u_i}^2 \, dx \right)^{\frac{1}{2}}\left( \int_{B_i}  \left( \nabla \Pi_i(0,a,\xi)  -\nabla \Pi_i(0,\dbar{u}_i,\xi)\right)^2 dx \right)  ^{\frac{1}{2}} \\
&= \left[ \int_{B_i} \abs{ \sum_{j,k=1}^3 \left( \frac{\partial \Pi_i^k(0,a,\xi)}{\partial x^j}  - \frac{\partial \Pi_i^k(0,\dbar{u}_i, \xi)}{\partial x^j}\right)}^2 \,dx \right]^{\frac{1}{2}} \\
&= \left[ \int_{B_i} \abs{ \sum_{j,k=1}^3 \left( \sum_{l=1}^3\frac{\partial \Pi_i^k}{\partial z_l}(0,a,\xi)\frac{\partial \xi^l}{\partial x^j}  - \frac{\partial \Pi_i^k}{\partial z_l}(0,\dbar{u}_i, \xi)\frac{\partial \xi^l}{\partial x^j}\right)}^2 \,dx \right]^{\frac{1}{2}} \\
&\leq \left[ \int_{B_i} c\abs{a - \dbar{u}_i}dx \right]^{\frac{1}{2}} \rightarrow 0 \text{ as } i \rightarrow \infty.
\end{align*}
\framebox{$IV_i \rightarrow 0$}
\begin{align*}
IV_i &= \varepsilon_i^{-1}\abs{ \int_{B_i} \nabla u_i \cdot \left[ \nabla \Pi_i(x,u_i,\xi) - \nabla \Pi_i(0,\dbar{u}_i,\xi)\right] \, dx} \\
&\leq \left[ 2 \underbrace{\int_{B_i} \abs{ \nabla \Pi_i (x,u_i , \xi) - \nabla \Pi_i(x,\dbar{u}_i,\xi)}^2\,dx}_{:=V_i} + \underbrace{ \int_{B_i} \abs{ \nabla \Pi_i(x,\dbar{u}_i,\xi) - \nabla \Pi_i(0,\dbar{u}_i,\xi)}^2 \, dx}_{:=VI_i} \right]^{\frac{1}{2}} \\
&= \left( 2(V_i + VI_i)\right)^{\frac{1}{2}}
\end{align*}
If we can show $V_i, VI_i \rightarrow 0$ then $IV_i \rightarrow 0$ and hence $II_i \rightarrow 0$. Using Lemma \ref{lem:diff between derivatives of Pi} and \eqref{inequal: bounds on derivatives of Pi} we have
\begin{align*}
V_i &= \int_{B_i} \abs{ \nabla \Pi_i (x,u_i , \xi) - \nabla \Pi_i(x,\dbar{u}_i,\xi)}^2 \, dx \\
&= \int_{B_i} \sum_{j,k=1}^3 \left| \sum_{l=1}^3 \delta_{lj}\left( \frac{\partial \Pi_i^k}{\partial x_l}(x,u_i,\xi) - \frac{\partial \Pi_i^k}{\partial x_l}(x,\dbar{u}_i,\xi) \right)  + \sum_{l=1}^3 \frac{\partial u_i^l}{\partial x^j} \frac{\partial \Pi_i^k}{\partial y_l}(x,u_i,\xi) \right. \\ &\qquad \left. + \sum_{l=1}^3 \frac{\partial \xi_l}{\partial x^j}\left( \frac{\partial \Pi_i^k}{\partial z_l}(x,u_i,\xi) - \frac{\partial \Pi_i^k}{\partial z_l}(x,\dbar{u}_i,\xi) \right) \right|^2dx \\
&\leq \int_{B_i} c \varepsilon_i^2 + c\varepsilon_i^2 + c \abs{u_i(x) - \dbar{u}_i}^2 dx + c\int_{\Omega_{\varepsilon_i}}\abs{\nabla u_i}^2 \, dx \\
&\leq \int_{B_i} 2 c \varepsilon_i^2 + c \delta_i^2 dx + c \varepsilon_i^2 \rightarrow 0. 
\end{align*}
Finally,
\begin{align*}
VI_i &= \int_{B_i} \abs{ \nabla \Pi_i(x,\dbar{u}_i, \xi) - \nabla \Pi_i(0, \dbar{u}_i, \xi)}^2 \, dx \\
&= \int_{B_i} \sum_{j,k=1}^3 \abs{ \sum_{l=1}^3 \delta_{jl}\frac{\partial \Pi_i^k}{\partial x_l}(x,\dbar{u}_i,\xi) + \sum_{l=1}^3\frac{\partial \xi_l}{\partial x^j}\left(\frac{\partial \Pi_i^k}{\partial z_l}(x,\dbar{u}_i,\xi) - \frac{\partial \Pi_i^k}{\partial z_l}(0,\dbar{u}_i,\xi) \right)}^2.
\end{align*}
Using Lemma \ref{lem:diff between derivatives of pi 2} we have 
\begin{equation*}
\abs{\frac{\partial \Pi_i^k}{\partial z_l}(x,y,z) - \frac{\partial \Pi_i^k}{\partial z_l}(0,y,z)} \leq c\abs{\nu_i(x) - \nu_i(0)} \leq c \lip{\varphi_i} \leq c \varepsilon_i.
\end{equation*}
Therefore
$$VI_i \leq \int_{B_i} c\varepsilon_i^2 + c\varepsilon_i \, dx  \rightarrow 0 \text{ as }i \rightarrow 0,$$
and thus we have shown
\begin{equation*}
\int_{\Omega_0} \nabla v \cdot \nabla  \xi \, dx = 0.
\end{equation*}

It now follows that $v^{\perp}$ and $v^{\top}$ extend by even and odd reflections to functions harmonic on $\mathcal{C}$. Thus, since $\dbar{v} = 0$ and $\int_{\Omega_0} \abs{\nabla v}^2 dx \leq 1$,  we have
\begin{align}
r^{-3}\int_{C(0,r)\cap\Omega_0}\abs{v}^2 dx &\leq c r^{2}\int_{\Omega_0} \abs{v}^2 \, dx \notag \\
&\leq c r^2  \int_{ \Omega_0} \abs{\nabla v}^2 \, dx \notag\\ &\leq cr^2.\label{ineq: L2 norm of v}
\end{align}
The first inequality follows from standard linear elliptic theory (similar to the proof of Lemma 2.2 in \cite{HardtLinKinderlehrer86}) and the second follows using Lemma \ref{lemma: Poincare constant}. The second inequality of \eqref{ineq: L2 norm of v} along with the Poincar\'{e} inequality and trace theory also implies
\begin{equation}\label{ineq: L2 norm of V on bdry}
r^{-4} \int_{C(0,r)\cap G_0} \abs{v}^2 d\mcH^2 \leq r^{-3} \norm{v}^2_{H^{\frac{1}{2}}(C(0,r) \cap G_0)} \leq cr^2.
\end{equation}
We now use Vitali convergence Theorem (Theorem \ref{thm: Vitali convergence}) to show that $\norm{v_i}_{L^{2}(\Omega_{\varphi_i})} \rightarrow \norm{v}_{L^2(\Omega_0)}$.
As $\widehat{v_i}$ converges weakly in $W^{1,2}(\mathcal{C})$ to $\widehat{v}$, we have that $\widehat{v_i}$ converges strongly in $L^2(\mathcal{C})$ (on a subsequence) to $\widehat{v}$. Hence we have 
\begin{equation*}
\widehat{v_i} \mathds{1}_{\Omega_{\varphi_i}}(x) \rightarrow v \mathds{1}_{\Omega_0}(x)
\end{equation*}
for $\mathcal{L}^3$ almost everywhere $x \in \mathcal{C}$, where $\mathds{1}_{\Omega_{\varphi_i}}$ and $\mathds{1}_{\Omega_0}$ are the indicator functions of $\Omega_{\varphi_i}$ and $\Omega_0$ respectively. Using the Sobolev embedding Theorem we have
\begin{align}\label{ineq: L3 norm of v_i}
\norm{v_i}_{L^3(\Omega_{\varphi_i})} \leq \norm{\widehat{v_i}}_{L^3(\mathcal{C})} \leq C'\norm{\widehat{v_i}}_{W^{1,2}(\mathcal{C})} \leq C'C\norm{v_i}_{W^{1,2}(\Omega_{\varphi_i})} \leq K.
\end{align}
Let $E \subset \mathcal{C}$ be an arbitrary measurable set. Then H\"{o}lders inequality and inequality \eqref{ineq: L3 norm of v_i} yields  
\begin{align*}
\int_{E}\abs{\widehat{v_i}\mathds{1}_{\Omega_{\varphi_i}}}^2 \, dx &\leq \left( \int_E \abs{\widehat{v_i} \mathds{1}_{\Omega_{\varphi_i}}}^3 \, dx\right)^{\frac{2}{3}} \mathcal{L}^3(E)^{\frac{1}{3}} \\
&\leq K^{\frac{2}{3}}\left( \mathcal{L}^3(E)\right)^{\frac{1}{3}}.  
\end{align*}
Therefore for any $\epsilon >0$ there exists a $\delta >0$  such that $\int_{E}\abs{\widehat{v_i}\mathds{1}_{\Omega_{\varphi_i}}}^2 dx \leq \epsilon$ for all $i$ and all $E \subset \mathcal{C}$ such that $\mathcal{L}^3(E) <\delta$. Hence $\abs{v_i\mathds{1}_{\Omega_{\varphi_i}}}^2$ is uniformly integrable. Thus by Vitali convergence Theorem
we get 
\begin{equation*}
\int_{\mathcal{C}}\abs{ \abs{\widehat{v_i}\mathds{1}_{\Omega_{\varphi_i}}}^2 - \abs{v\mathds{1}_{\Omega_0}}^2} dx \rightarrow 0.
\end{equation*}
Therefore for any $\theta <r <1$ we have 
\begin{equation} \label{ineq:convergence of v_i in norm}
\abs{ \fint_{\Omega_{\varphi_i} \cap C(0,r)} \abs{v_i}^2 \, dx - \fint_{\Omega_0 \cap C(0,r)} \abs{v}^2 \,  dx } \leq  c \theta^2 \leq c r^2
\end{equation}
for sufficiently large $i$. We also have $\norm{v_i}_{L^2(G_{\varphi_i})} \rightarrow \norm{v}_{L^2(G_0)}$, hence for sufficiently large $i$ we have
\begin{equation}\label{ineq: convergence of v_i in norm on bdry}
\abs{ \fint_{G_{\varphi_i} \cap C(0,r)} \abs{v_i}^2 d\mcH^2 - \fint_{G_0 \cap C(0,r)} \abs{v}^2 d\mcH^2 } \leq  c \theta^2 \leq c r^2.
\end{equation}
Combining \eqref{ineq: L2 norm of v}, \eqref{ineq: L2 norm of V on bdry}, \eqref{ineq:convergence of v_i in norm} and \eqref{ineq: convergence of v_i in norm on bdry} we arrive at
\begin{equation}\label{ineq: bound on u_i - u_i 1}
\fint_{\Omega_{\varphi_i}\cap C(0,r) } \abs{u_i - \dbar{u}_i}^2 dx \leq \varepsilon_i^2\left( cr^2 + \fint_{\Omega_0 \cap C(0,r)} \abs{v}^2 \, dx\right) \leq c \varepsilon_i^2 r^2
\end{equation}
and similarly
\begin{equation}\label{ineq: bound on u_i - u_i 2}
\fint_{G_{\varphi_i}\cap C(0,r)} \abs{u_i - \dbar{u}_i}^2d \mcH^2 \leq c \varepsilon_i^2  r^2.
\end{equation}

We now use the Hybrid inequality (Lemma \ref{lem:Hybrid Inequality}) and the scaling discussed in section \ref{subsec: scaling} in order to to contradict our assumption that 
$\frac{1}{\theta} \int_{C(0,\theta) \cap \Omega} \abs{\nabla u_i}^2 dx > \theta \varepsilon_i^2$. 
For each $i$ we apply the Hybrid inequality to the scaled function $(u_i)_{4\theta}$ to obtain
\begin{align*}
\frac{1}{2\theta} &\int_{C(0,\theta) \cap \Omega_{\varphi_i}} \abs{\nabla u_i}^2 \, dx   \\&\leq \lambda\left(  \liptwo{(\varphi_i)_{4\theta}}^2 + \frac{1}{4\theta} \int_{C(0,4\theta)\cap \Omega_{\varphi_i}}\abs{\nabla u_i}^2dx \right) \\ &\quad + \lambda^{-q} c \left( \lip{(\varphi_i)_{4\theta}}^2 + \fint_{C(0,4\theta)\cap\Omega_{\varphi_i}}\abs{u_i - \dbar{u}_i}^2 \, dx + \fint_{C(0,4\theta)\cap G_{\varphi_i}}\abs{u_i - \dbar{u}_i}^2 \, d\mcH^2 \right) \\ &\quad + c \liptwo{(\varphi_i)_{4\theta}}^2 \\
&\leq \lambda \left[ (4 \theta)^2 \liptwo{\varphi_i}^2 + \frac{1}{4\theta}\int_{C(0,4\theta) \cap \Omega_{\varphi_i}} \abs{\nabla u_i}^2 \, dx \right] \\ &\quad + \lambda^{-q}c \left( (4\theta)^2\lip{\varphi_i}^2 + \fint_{C(0,4\theta)\cap \Omega_{\varphi_i}}\abs{u_i - \dbar{u}_i}^2 \, dx + \fint_{C(0,4\theta)\cap G_{\varphi_i}}\abs{u_i - \dbar{u}_i}^2 \, d\mcH^2 \right) \\ &\quad + c (4\theta)^2\liptwo{\varphi_i}.
\end{align*} 
Iterating this $(k-1)$ more times (where $k$ is an integer to be chosen later) gives
\begin{align*}
\frac{1}{2\theta} \int_{C(0,\theta) \cap \Omega_{\varphi_i}} \abs{\nabla u_i}^2 \, dx &\leq \lambda^k(4^k\theta)^{-1} \int_{\Omega_{\varphi_i}\cap C(0,4^k\theta)}\abs{\nabla u_i}^2 \, dx \\ 
&\qquad + \sum_{j=1}^k \lambda^j (4^j \theta)^2 \liptwo{\varphi_i}^2 \\ &\qquad+ \sum_{j=1}^k \lambda^{-q} \lambda^{j-1} \left( (4^j\theta)^2 \lip{\varphi_i}^2 + \fint_{C(0,4^j\theta)\cap \Omega_{\varphi_i}}\abs{u_i - \dbar{u}_i}^2 \, dx  \right. \\ &\qquad \qquad \qquad \qquad \qquad  \left.+ \fint_{C(0,4^j\theta)\cap G_{\varphi_i}}\abs{u_i - \dbar{u}_i}^2 \, d\mcH^2 \right) \\ 
&\qquad + \sum_{j=1}^k \lambda^{j-1}(4^j \theta)^2 \liptwo{\varphi_i}^2 .
\end{align*}
Let $k$ be an integer such that $4^k \theta \leq 1 \leq 4^{k+1} \theta$, then using \eqref{ineq: bound on u_i - u_i 1} and \eqref{ineq: bound on u_i - u_i 2} we have for $i$ large enough 
\begin{align*}
\frac{1}{2\theta} \int_{C(0,\theta) \cap \Omega_{\varphi_i}} \abs{\nabla u_i}^2 \, dx &\leq 4\lambda^k\varepsilon_i^2 + \sum_{j=1}^k \theta^2 \varepsilon_i^2 (16 \lambda)^j \\ &\qquad + \lambda^{-q-1}\varepsilon_i^2 \sum_{j=1}^k \left[ (16 \lambda)^j\theta^2 + 2 c \theta^2 (16\lambda)^j \right] + \lambda^{-1} \varepsilon_i^2 \theta^2 \sum_{j=1}^k(16 \lambda)^j \\
&\leq 4 \lambda^k \varepsilon_i^2 + \lambda^{-l} c \theta^2 (1 - 16\lambda)^{-1} \varepsilon_i^2  
\end{align*}
where $l =\max{\{q+1,1\}}$ and $c$ has absorbed all other constants. Set $\lambda = \theta^{\frac{3}{k}}$, then providing $\theta < \frac{1}{4}$ we have $\lambda^k = \theta^3 < \frac{\theta}{16}$. Next, as $k\rightarrow \infty$ as $\theta \rightarrow 0$, we may fix $\theta < \frac{1}{4}$ sufficiently small such that 
$$16c \theta < \theta^{\frac{3l}{k}}.$$
Note that for $\theta$ sufficiently small we have $\lambda < \frac{1}{32}$ which implies $(1-16\lambda)^{-1} \leq 2$ then
\begin{align*}
\lambda^{-l}c\theta^2 (1-16\lambda)^{-1} &= \theta^{- \frac{3l}{k}}c\theta^2(1-16\lambda)^{-1}  \\
&\leq \frac{c \theta^2}{16 c \theta}(1- 16\lambda)^{-1} \\
&\leq \frac{\theta}{8}.
\end{align*}
Therefore
\begin{equation*}
\frac{1}{2\theta}\int_{C(0,\theta) \cap \Omega_{\varphi_i}} \abs{ \nabla u_i}^2 \, dx < \left( \frac{\theta}{4} + \frac{\theta}{8}\right) \varepsilon_i^2 = \frac{3\theta}{8} \varepsilon_i^2,
\end{equation*}
contradicting our choice of $u_i,\varphi_i$.

\end{proof}

\begin{boxtheorem}[Energy Decay]\label{Thm: energy decay}
Suppose $\Omega$ is a $C^2$ domain in $\RR^3$ and take $a \in \partial \Omega$. If $u \in \mcU$ is a minimizer of $\mathbb{G}$ in $\mathcal{U}$ with $\int_{\Omega \cap C(a,R)}\abs{ \nabla u}^2 dx \leq \varepsilon^2 R$, then
\begin{equation*}
\int_{\Omega\cap C(a,r)} \abs{\nabla u }^2 \, dx \leq r^2R^{-1} \theta^{-2} \max \{ \varepsilon^2 , c(R + R^2) \} 
\end{equation*} 
for $0<r \leq R$, a suitable $\theta\in (0,1)$ and $c$ a constant depending on $\Omega$.
\end{boxtheorem}

\begin{proof}
We apply inequality \eqref{inequ: Energy Improvement} to the scaled function $v^1(x) = u_{\theta R ,a}(x)$ to get that for a $\theta\in (0,1)$ as in Lemma~\ref{lem:Energy Improvement} we have:
\begin{align*}
\frac{1}{R\theta} \int_{C(a,R\theta)\cap\Omega} \abs{\nabla u(x)}^2 dx 
&= \int_{C(a,1) \cap \Omega} \abs{\nabla u_{R\theta,a}(x)}^2 dx \\
&=\frac{1}{\theta}\int_{C(a,\theta) \cap \Omega} \abs{\nabla u_{R,a}(x)}^2 dx \\
&\leq \theta \max \{ \int_{C(a,1)\cap \Omega} \abs{\nabla u_{R,a}(x)}^2 dx , c\left( \lip{\varphi_{R,a}} + \liptwo{\varphi_{R,a}}\right)\} \\ 
&\leq \theta \max \{ \int_{C(a,1)\cap \Omega} \abs{\nabla u_{R,a}(x)}^2 dx , c\left( R\lip{\varphi} + R^2\liptwo{\varphi}\right)\}. 
\end{align*}
Then, inductively,  for $k \in \NN$ we get that 
$$(\theta^{k-1} R)^{-1} \int_{C(a,\theta^{k-1}R)\cap\Omega}\abs{\nabla u }^2 dx \leq \theta^{k-1} \max \{\varepsilon^2 , c \left( R \lip{\varphi} + R^2 \liptwo{\varphi}\right)\}$$

hence
\begin{align*}
(\theta^kR)^{-1} \int_{C(a,\theta^k R) \cap \Omega}\abs{\nabla u}^2 dx &= \frac{1}{\theta}\int_{C(a,\theta)\cap\Omega}\abs{\nabla u_{\theta^{k-1}R,a}}^2dx \\
&\leq \theta \max \left\{ \int_{C(a,1)} \abs{\nabla u_{\theta^{k-1}R,a}}^2 dx , c\left( \lip{\varphi_{\theta^{k-1}R,a}} + \liptwo{\varphi_{\theta^{k-1}R,a}}\right) \right\} \\
&\leq \theta \max \{ \left[ \theta^{k-1} \max \left\{\varepsilon^2 , c \left( R \lip{\varphi} + R^2 \liptwo{\varphi}\right)\} \right],\right. \\
&\left. \qquad \qquad \qquad c\left( (\theta^{k-1}R)\lip{\varphi} + (\theta^{k-1}R)^2\liptwo{\varphi}\right) \right\} \\
&= \theta^k \max \left\{ \varepsilon^2 , c \left( R \lip{\varphi} + R^2 \liptwo{\varphi}\right)\right\}.
\end{align*}
Hence the inequality 
\begin{equation*}
(\theta^kR)^{-1} \int_{C(a,\theta^kR) \cap \Omega}\abs{\nabla u}^2 dx \leq \theta^k \max \left\{ \varepsilon^2 , c \left( R \lip{\varphi} + R^2 \liptwo{\varphi}\right)\right\} 
\end{equation*}
holds for all $k \in \NN$. Next, choose $k$ such that $\theta^{k+1}R \leq r < \theta^k R$. Then
\begin{align*}
r^{-1}\int_{C(a,r)\cap \Omega}\abs{\nabla u}^2 dx &\leq \theta^{-1}(\theta^k R)^{-1} \int_{C(a,\theta^k R) \cap \Omega}\abs{\nabla u}^2 dx \\
&\leq \theta^{-1} \theta^k  \max \left\{ \varepsilon^2 , c \left( R \lip{\varphi}^2 + R^2 \liptwo{\varphi}^2\right)\right\} \\
&\leq \theta^{-1} (\theta^{-1}R^{-1}r)  \max \left\{ \varepsilon^2 , c \left( R \lip{\varphi}^2 + R^2 \liptwo{\varphi}^2\right)\right\} \\
&\leq \theta^{-2}R^{-1}r \max\left\{\varepsilon^2 , c(R + R^2) \right\}.
\end{align*}

\end{proof}

\bigskip \begin{proof}{\bf [of Theorem~\ref{thm:preg}]}

We note that we just need to look into partial regularity on the boundary, because in the interior of $\Omega$ the energy behaves like that of the standard harmonic map. Indeed, let $\bar n$ be a minimizer of $\mathbb{G}$ and  take an arbitrary open set $\omega$ with $\bar\omega\subset\Omega$. Let $n_\omega$ be  a minimizer of $\int_{\omega} |\nabla n|^2\,dx$ in $W^{1,2}(\omega;\mathbb{S}^2)$ with boundary condition $\bar n|_{\omega}$. Then defining $\tilde n(x):=\left\{\begin{array}{ll} n(x) & \textrm{for }x\in\Omega\setminus\omega\\
n_\omega(x) &\textrm{ for }x\in\omega\end{array}\right.$ we have that $\tilde n$ is a minimizer of $\mathbb{G}$ and thus all the arguments for harmonic maps can be applied to it in $\omega$.

The main point of our work in this last section has been to obtain the energy decay near the boundary, namely Theorem~\ref{Thm: energy decay} before. For a point where the rescaled energy is small enough the standard Morrey lemma shows that the estimate thus obtained implies H\"older regularity. Standard covering arguments show then the set of singularitiy points, those  where the renormalized energy is not small, is  of zero one-dimensional Hausorf measure (see for instance \cite{HardtLin89}). 
\end{proof}
\begin{appendices}

\section{Some Lemmas for Partial Regularity}

We prove two lemmas that will be useful in the proof of the partial regularity.  They are standard results for a fixed domain but we need them to hold with certain constants independent of the domains.

\smallskip
 The first is an extension lemma and it is standard (see for instance  \cite{MeasureandFineProps} Chapter 4), but we reproduce a proof for the readers convenience. Its main point is to show that the extension map can be chosen in a manner that depends only on the upper bound of the Lipschitz constant of the boundary. 

\begin{lemma}\label{lemma:extension lemma}
Let $\varphi:\RR^2 \rightarrow \RR$ be a Lipschitz map such that $\varphi(0)= \abs{ \nabla \varphi(0)}^2 = 0$ and $\lip{\varphi} \leq \frac{1}{5}$. Then if $f \in W^{1,2}(\Omega_{\varphi})$, there exists $\widehat{f} \in W^{1,2}(\mathcal{C})$ such that $\widehat{f}=f$ on $\Omega_\varphi$ and
\begin{equation*}
\norm{\widehat{f}}_{W^{1,2}(C(0,1))} \leq c \norm{f}_{W^{1,2}(\Omega_{\varphi})}
\end{equation*} 
where $c$ is independent of $\varphi,f$ and we denote $\mathcal{C}:= \{x \in \RR^3: \abs{(x^1,x^2)}<1 \text{ and } -1<x^3 <\frac{1}{2}\}$ respectively $\Omega_{\varphi} := \{(x^1,x^2,x^3)\in C(0,1) : x^3 < \varphi(x^1,x^2)\} $.
\end{lemma}
\begin{proof}
Fix $\varphi$ and define
\begin{align*}
U:= \mathcal{C}\setminus \overline{\Omega}_\varphi.
\end{align*} 
For simplicity we write $x':=(x^1,x^2)$. 
First suppose that $f \in C^1(\overline{\Omega}_\varphi)$ and set 
\begin{equation*}
\begin{cases}
\widehat{f}(x)=f(x) \text{ if } x \in \overline{\Omega_{\varphi}}, \\
\widehat{f}(x)=f(x',-x^3 + 2\varphi(x')) \text{ if } x \in \overline{U}.
\end{cases}
\end{equation*}
We note that $f = \widehat{f}$ on $\Omega_\varphi \cap \mathcal{C}$ and we claim that 
\begin{equation*}
\norm{\widehat{f}}_{W^{1,2}(U)} \leq c_1 \norm{f}_{W^{1,2}(\Omega_\varphi)}.
\end{equation*}
To see this let $\psi \in C^1_c(U)$ and $\{\varphi_k\}_{k=1}^\infty$ be a sequence of $C^\infty$ functions such that 
\begin{equation*}
\begin{cases}
\varphi_k \leq \varphi \\
\varphi_k \rightarrow \varphi \text{ uniformly} \\
D\varphi_k \rightarrow D\varphi \text{ almost everywhere} \\
\sup_{k}\norm{D\varphi_k}_{L^\infty} < \infty.
\end{cases}
\end{equation*}
Then for $j = 1,2$ we have 
\begin{align*}
\int_{U} \widehat{f} \frac{\partial \psi}{\partial x^j} dx &= \int_{U} f(x',2\varphi(x') - x^3) \frac{\partial \psi}{\partial x^j} dx \\
&= \lim_{k\rightarrow \infty} \int_{U} f(x',2 \varphi_k(x')-x^3) \frac{\partial \psi}{\partial x^j} dx \\
&= -\lim_{k \rightarrow \infty} \int_{U} \left( \frac{\partial f}{\partial x^j}(x',2\varphi_k(x')-x^3) + 2 \frac{\partial f}{\partial x^3}(x',2\varphi_k(x')-x^3)\frac{\partial \varphi_k}{\partial x^j}(x') \right) \psi dx \\
&= -\int_{U} \left( \frac{\partial f}{\partial x^j}(x',2\varphi(x')-x^3) + 2 \frac{\partial f}{\partial x^3}(x',2\varphi(x')-x^3)\frac{\partial \varphi}{\partial x^j}(x') \right) \psi dx  
\end{align*}
In the same way we can compute
$$\int_{U} \widehat{f} \frac{\partial \psi}{\partial x^3}dx = \int_{U} \frac{\partial f}{\partial x^3}(x',2\varphi(x')-x^3) \psi dx.$$
Then, by a change of variables and using the fact $\lip{\varphi}\leq \frac{1}{5}$ we get
\begin{equation} \label{ineq: bound on comp sup f}
\int_{U} \abs{Df(x',2 \varphi(x') - x^3)}^2 dx \leq c_1 \int_{\Omega_\varphi} \abs{Df}^2 dx,
\end{equation}
for some appropriate $c_1$ that is independent of $f$ and $\varphi$. This proves the claim and moreover also shows that $\norm{\widehat{f}}_{W^{1,2}(\mathcal{C})} \leq c_1 \norm{f}_{W^{1,2}(\Omega_\varphi)}$. The result for a general $f \in W^{1,2}(\Omega_\varphi)$ can now be achieved via density.
\end{proof}

\bigskip
For the proof of the next lemma we will use Vitali convergence theorem which we state here for the readers convenience 

\begin{theorem}[Vitali Convergence Theorem] \label{thm: Vitali convergence}
Let $\left( X , \mathcal{F}, \mu \right)$ be a positive measure space. Let $f_n:X\to\R,\forall n\in\N$ and $f:X\to\R$ be measurable functions such that: 
\begin{enumerate}[a.]
\item $\mu(X) <\infty$
\item $\{f_n\}_{n\in\N}$ is uniformly integrable
\item $f_n(x) \rightarrow f(x)$ almost everywhere as $n \rightarrow \infty$ \item $\abs{f(x)} <\infty$ almost everywhere
\end{enumerate}
then the following hold 
\begin{enumerate}[i.]
\item $f \in L^1(\mu)$
\item $\lim_{n \rightarrow \infty} \int_X \abs{ f_n -f} d \mu = 0$
\end{enumerate}
\end{theorem}

We now use this and lemma \ref{lemma:extension lemma} to prove the following Poincar\'{e} inequality. Its main point is to check that the constant in the inequality can be chosen uniformly for all the domains we consider. 

\begin{lemma}\label{lemma: Poincare constant}
Let $\varphi:\RR^2 \rightarrow \RR$ be a Lipschitz map such that $\varphi(0)= \abs{ \nabla \varphi(0)}^2 = 0$ and $\lip{\varphi} \leq \frac{1}{2}$. Then there exists $C>0$, independent of $\varphi$, such that for all $u \in W^{1,2}(\Omega_{\varphi})$ we have the Poincar\'{e} type inequality
\begin{equation*}
\int_{\Omega_{\varphi}}\abs{ u - \dbar{u}}^2 dx \leq C \int_{\Omega_\varphi} \abs{\nabla u}^2 dx 
\end{equation*}
where $\dbar{u} := \mcH^2(G_{\varphi})^{-1} \int_{G_\varphi}u d\mcH^2$.
\end{lemma}
\begin{proof}
We argue by contradiction and  assume that for all $i\in\N$ there exist $\varphi_i$,$u_i \in W^{1,2}(\Omega_{\varphi_i})$ such that 
\begin{equation*}
\frac{\int_{\Omega_{\varphi_i}} \abs{\nabla u_i}^2 dx}{\int_{\Omega_{\varphi_i}}\abs{u_i - \dbar{u}_i}^2dx} < \frac{1}{i}.
\end{equation*}
Define the functions $v_i:= \left( u_i - \dbar{u}_i\right)\left( \int_{\Omega_{\varphi_i}}\abs{u_i - \dbar{u}_i}^2dx\right)^{-\frac{1}{2}}$. Then we have 
\begin{equation*}
\int_{\Omega_{\varphi_i}}\abs{\nabla v_i}^2dx \leq \frac{1}{i} \qquad \text{    and    }\qquad \int_{\Omega_{\varphi_i}}\abs{v_i}^2dx = 1.
\end{equation*}
As $\abs{\nabla \varphi_i} <\frac{1}{2}$, there exists a subsequence (not relabelled) such that $\varphi_i \rightarrow \varphi$ uniformly for some $\varphi$.

By Lemma \ref{lemma:extension lemma} there exist functions $\widehat{v}_i \in W^{1,2}(\mathcal{C})$ such that $\norm{\widehat{v}_i}_{W^{1,2}(\mathcal{C})} \leq c \norm{v_i}_{W^{1,2}(\Omega_{\varphi_i})}$ and by inspecting the proof there exists constant $c_1$ such that $\norm{\nabla \widehat{v}_i}_{L^{2}(\mathcal{C})} \leq c_1 \norm{\nabla v_i}_{L^{2}(\Omega_{\varphi_i})}$. It follows that there exists $\widehat{v} \in W^{1,2}(\mathcal{C})$ such that $\widehat{v}_i$ converge weakly in $W^{1,2}$ to $\widehat{v}$, but as $\norm{\nabla \widehat{v}_i}_{L^{2}(\mathcal{C})} \rightarrow 0$ we have that $\widehat{v} \equiv k$ for some constant $k$. We now investigate the behaviour of $\widehat{v}$ on $G_\varphi$ where we recall that $G_{\varphi} := \partial \Omega_{\varphi} \setminus \partial C(0,1)$. We have
\begin{align*}
\abs{\int_{G_{\varphi_i}}v_i d\mcH^2 - \int_{G_\varphi} \widehat{v}d\mcH^2 } &\leq \abs{ \int_{G_{\varphi_i}} (v_i - \widehat{v})\, d \mcH^2} + \abs{ \int_{G_{\varphi_i}}\widehat{v}\, d\mcH^2 - \int_{G_\varphi} \widehat{v}\, d\mcH^2 } \\
&\leq \norm{v_i - \widehat{v}}_{H^{\frac{1}{2}- \varepsilon}(G_{\varphi_i})} + \abs{ \int_{G_{\varphi_i}}\widehat{v}\, d\mcH^2 - \int_{G_\varphi} \widehat{v}\, d\mcH^2 } \\
&\leq \norm{v_i - \widehat{v}}_{H^{1- \varepsilon}(\Omega_{\varphi_i})} + \abs{ \int_{G_{\varphi_i}}\widehat{v}\,  d\mcH^2- \int_{G_\varphi} \widehat{v}d\mcH^2 }\\
&\leq \norm{\hat v_i - \widehat{v}}_{H^{1- \varepsilon}(\mathcal{C})} + \abs{ \int_{G_{\varphi_i}}\widehat{v}\,  d\mcH^2- \int_{G_\varphi} \widehat{v}\, d\mcH^2 }\
\end{align*}
As we have $\norm{\widehat{v}_i}_{H^1(\mathcal{C})}$ bounded uniformly we have that $\widehat{v}_i \rightarrow \widehat{v}$ in $H^{1-\varepsilon}(\mathcal{C})$ and since also $\varphi_i\to\varphi$ uniformly  the above tends to $0$. Then 
\begin{equation} \label{eq: poincare contradiction 1}
0=\int_{G_{\varphi_i}} v_i d\mcH^2 \rightarrow \int_{G_\varphi} \widehat{v}d\mcH^2,
\end{equation}
which implies that $\widehat{v}\equiv 0$.

However, we claim that 
\begin{equation}\label{eq: poincare contradiction 2}
1 = \lim_{i \rightarrow \infty} \int_{\Omega_{\varphi_i}} \abs{v_i}^2 dx = \int_{\Omega_\varphi} \abs{v}^2 dx
\end{equation}
which would contradict $\widehat{v}\equiv 0$. 

To prove \eqref{eq: poincare contradiction 2} we use Vitali convergence Theorem. 
First as $\widehat{v}_i$ converges strongly in $L^2(\mathcal{C})$ to $\widehat{v}$ we have 
\begin{equation*}
\widehat{v}_i(x) \mathds{1}_{\Omega_{\varphi_i}}(x) \rightarrow \widehat{v}(x) \mathds{1}_{\Omega_{\varphi}}(x)
\end{equation*}
for almost every $x \in \mathcal{C}$.

We now show that the sequence of functions $\{\abs{\widehat{v}_i \mathds{1}_{\Omega_{\varphi_i}}}^2\}$ are uniformly integrable. Note that  by the Sobolev inequalities, there exists $c_2$ such that $\norm{\widehat{v}_i}_{L^3(\mathcal{C})} \leq c_2 \norm{\widehat{v}_i}_{W^{1,2}(\mathcal{C})}$ hence as $\widehat{v}_i$ is bounded in $W^{1,2}(\mathcal{C})$ it is also a bounded sequence in $L^3(\mathcal{C})$, say $\norm{\widehat{v}_i}_{L^3(\mathcal{C})} \leq c_3$. Fix $\varepsilon >0$ we want to show that there exists $\delta>0$   such that $\int_E\abs{\widehat{v}_i \mathds{1}_{\Omega_{\varphi_i}}}^2 dx < \varepsilon$ whenever $E\subset \mathcal{C}$ and $\mathcal{L}^3(E) < \delta$ . With this in mind for $E\subset \mathcal{C}$ we estimate 
\begin{align*}
\int_E\abs{\widehat{v}_i \mathds{1}_{\Omega_{\varphi_i}}}^2 dx &\leq \left(\int_E\abs{\widehat{v}_i \mathds{1}_{\Omega_{\varphi_i}}}^3 dx\right)^{\frac{2}{3}}\left( \int_E dx\right)^{\frac{1}{3}} \\
&\leq \norm{\widehat{v}_i}^2_{L^3(\mathcal{C})} \mathcal{L}^3(E)^{\frac{1}{3}} \\
&\leq c_3 \mathcal{L}^3(E)^{\frac{1}{3}} 
\end{align*}
thus by setting $\delta = \varepsilon^3/c_3$ we have shown $\{\abs{\widehat{v}_i \mathds{1}_{\Omega_{\varphi_i}}}^2\}$ are uniformly integrable. We can now conclude using Vitali convergence Theorem that 
\begin{equation*}
\lim_{i \rightarrow \infty} \int_{\mathcal{C}} \abs{ \abs{\widehat{v}_i \mathds{1}_{\Omega_{\varphi_i}}}^2 - \abs{\widehat{v} \mathds{1}_{\Omega_{\varphi}}}^2} dx = 0
\end{equation*} 
which implies \eqref{eq: poincare contradiction 2} and completes the proof.
\end{proof}

\section{Some bounds on the Projection \eqref{def:proj onto tangent space}}
We prove some bounds on the derivatives of the projection $\Pi_i : \Omega_{\varphi_i} \times U \times \RR^3 \rightarrow \RR^3$ defined by 
\begin{equation*} 
\Pi_i(x,y,z) = \frac{ \left[ (y \times \nu_i(x)) \otimes (y \times \nu_i(x))\right]z}{\abs{y \times \nu_i(x)}^2} + y,
\end{equation*}
where $\Omega_{\varphi_i}$ and $U$ are as in the proof of Lemma \ref{lem:Energy Improvement}.

\begin{lemma}\label{lemma:appendix bounds on derivatives of Pi}
Providing $U$ is sufficiently small then there exists $C>0$ such that for large $i$ the projection $\Pi_i$ satisfies the bounds
\begin{align*}
\abs{\frac{\partial \Pi_i(x,y,z)}{\partial x^j}} < C \varepsilon_i^2\abs{z}, \qquad
\abs{\frac{\partial \Pi_i(x,y,z)}{\partial z^j}} \leq 1, \qquad
\abs{\frac{\partial \Pi_i(x,y,z)}{\partial y^j}} < C\abs{z}.
\end{align*}
\end{lemma}
\begin{proof}
Writing $y=(y^1,y^2,y^3)$ a straight forward calculation yields
\begin{align*}
\frac{\partial \Pi_i}{\partial y^j}(x,y,z) &= \frac{z \cdot \left( (\delta_{1j},\delta_{2j},\delta_{3j})\times\nu_i(x)\right)(y \times \nu_i(x))}{\abs{y \times \nu_i(x)}^2} \\ &\qquad
+ \frac{z \cdot (y \times \nu_i(x))\left( (\delta_{1j},\delta_{2j},\delta_{3j})\times \nu_i(x)\right)}{\abs{y \times \nu_i(x)}^2} \\
&\qquad -
\frac{2\left( (y \times \nu_i(x))\cdot\left( (\delta_{1j},\delta_{2j},\delta_{3j})\times \nu_i(x)\right) \right) \left( z \cdot (y \times \nu_i(x)) \right)(y \times \nu_i(x))}{\abs{y \times \nu_i(x)}^4} \\
&\qquad \qquad  + (\delta_{1j},\delta_{2j},\delta_{3j}) 
\end{align*}
\begin{align*}
\frac{\partial \Pi_i}{\partial z^j}(x,y,z) &= \frac{\left[(y \times \nu_i(x)) \otimes (y \times \nu_i(x)) \right](\delta_{1j},\delta_{2j},\delta_{3j})}{\abs{y \times \nu_i(x)}^2}\\
\end{align*}	
\begin{align*}
\frac{\partial \Pi_i}{\partial x^j}(x,y,z) &= \frac{ z \cdot \left( y \times \frac{\partial \nu_i(x)}{\partial x^j} \right)(y \times \nu_i(x)) }{\abs{y \times \nu_i(x)}^2} +
\frac{z \cdot (y \times \nu_i(x)) \left(y \times \frac{\partial\nu_i(x)}{\partial x^j}\right)}{\abs{y \times \nu_i(x)}^2} \\ &\qquad -
\frac{2\left( (y \times \nu_i(x)) \cdot \left(y \times \frac{ \partial \nu_i(x)}{\partial x^j}\right) \right)\left( z \cdot (y \times \nu_i(x)) \right)(y \times \nu_i(x))}{\abs{y \times \nu_i(x)}^4}
\end{align*}
where \begin{equation*}
\frac{\partial \nu}{\partial x^j} = \left[ \frac{\left(-\varphi_{x^1x^j},-\varphi_{x^2x^j},0 \right)}{\left(\varphi_{x^1}^2 +\varphi_{x^2}^2 + 1 \right)^{\frac{1}{2}}} - \frac{\left(\varphi_{x^1x^j}\varphi_{x^1} + \varphi_{x^2x^j}\varphi_{x^2}\right) \left( -\varphi_{x^1} , -\varphi_{x^2}, 1  \right) }{\left(\varphi_{x^1}^2 +\varphi_{x^2}^2 + 1 \right)^{\frac{3}{2}}} \right].
\end{equation*} 

We observe that if $U$ is sufficiently small and $i$ sufficiently large then there exists constants $c_1,c_2>0$ such that $c_1 < \abs{y \times \nu_i(x)}<c_2$. This, along with the inequality \eqref{ineq: lip2 phi_i bound}, gives
\begin{align*}
\abs{\frac{\partial \Pi_i(x,y,z)}{\partial x^j}} < C \varepsilon_i^2\abs{z}, \qquad
\abs{\frac{\partial \Pi_i(x,y,z)}{\partial z^j}} \leq 1, \qquad
\abs{\frac{\partial \Pi_i(x,y,z)}{\partial y^j}} < C\abs{z}.
\end{align*}
\end{proof}
We next prove a bound on the difference of two projections onto two different lines.

\begin{lemma}\label{lem:diff between derivatives of Pi}
Let $y_1,y_2 \in U$, $x \in \Omega_{\varphi_i}$ and $z \in \RR^3$. Then we have the bound
 \begin{equation*}
\abs{ \frac{\partial \Pi_i}{\partial z^j}(x,y^1,z) - \frac{\partial \Pi_i}{\partial z^j}(x,y^2,z)} \leq  \abs{y^1 -y^2}.
\end{equation*}
\end{lemma}

\begin{proof}
First, by writing $\Lambda_j = (\delta_{1j}, \delta_{2j},\delta_{3j})$ we have
\begin{align*}
\frac{\partial \Pi_i}{\partial z^j}(x,y,z) &= \frac{ (y \times \nu_i(x)) \otimes (y \times \nu_i(x))\Lambda_j}{\abs{y \times \nu_i(x)}^2} \\
&=  A(y,x)\Lambda_j,
\end{align*}
where $A(x,y) = \frac{(y \times \nu_i(x)) \otimes (y \times \nu_i(x))}{\abs{y \times \nu_i(x)}^2}\in \mathbb{M}^{3\times 3}$.  We consider the linear map $h:\RR^3 \rightarrow \RR^3$ 
$$z \mapsto (A(y^1,x)- A(y^2,x))z,$$
and calculate the operator norm of $h$. Note that the map $v\mapsto A(x,y)v$ is the projection of $v$ onto the line spanned by $(y \times \nu_i(x))$.We assume that the plane spanned by $(y^1 \times \nu_i(x))$ and $(y^2 \times \nu_i(x))$ to be $\RR^2$ with
$$((y^1 \times \nu_i(x)) = (1,0,0) \text{ and }(y^2 \times \nu_i(x)) = (\cos(\tau),\sin(\tau),0),$$
where $\tau$ is the angle between $(y^1 \times \nu_i(x))$ and $(y^2 \times \nu_i(x))$. Then 
\begin{equation*}
A(y^1,x) - A(y^2,x) = \left[
\begin{matrix}
1- \cos^2(\tau)   & -\cos(\tau)\sin(\tau) & 0 \\
-\cos(\tau)\sin(\tau) & -\sin^2(\tau) & 0 \\
0 & 0 & 0
\end{matrix}
\right].
\end{equation*}
We can find the operator norm of this by finding the largest eigenvalue. A calculation yields
\begin{equation*}
\norm{h}_{op} = \abs{\sin(\tau)}.
\end{equation*}
Using elementary geometry we have
\begin{align*}
\sin(\tau) &\leq \frac{\abs{(y^1 \times \nu_i(x))-(y^2 \times \nu_i(x))}}{\abs{y^1 \times \nu_i(x)}} \\
&\leq  \abs{ y^1 - y^2}.
\end{align*}
Therefore
\begin{equation*}
\abs{ \left( \frac{(y^1 \times \nu_i(x)) \otimes (y^1 \times \nu_i(x))}{\abs{(y^1 \times \nu_i(x))}^2} - \frac{(y^2 \times \nu_i(x)) \otimes (y^2 \times \nu_i(x))}{\abs{(y^2 \times \nu_i(x))}^2}\right)\Lambda_j} \leq \abs{\Lambda_j}\abs{\sin(\tau)} \leq \abs{y^1 - y^2}.
\end{equation*}
\end{proof}
Using similar reasoning one can also prove the following.
\begin{lemma}\label{lem:diff between derivatives of pi 2}
There exists $C>0$ such that
\begin{equation*}
\abs{\frac{\partial \Pi_i^k}{\partial z_l}(x,y,z) - \frac{\partial \Pi_i^k}{\partial z_l}(0,y,z)} \leq C\abs{\nu_i(x) - \nu_i(0)},
\end{equation*}
for all $x\in \Omega_{\varphi_i}$, $y \in U$ and $z \in \RR^3$.
\end{lemma}
\begin{proof}
We first calculate that
\begin{align*}
&\abs{\frac{\partial \Pi_i^k}{\partial z_l}(x,y,z) - \frac{\partial \Pi_i^k}{\partial z_l}(0,y,z)} = \\ &\quad  \abs{ \frac{ (y \times \nu_i(x)) \otimes (y \times \nu_i(x))\Lambda_j}{\abs{y \times \nu_i(x)}^2} - \frac{ (y \times \nu_i(0)) \otimes (y \times \nu_i(0))\Lambda_j}{\abs{y \times \nu_i(0)}^2} }.
\end{align*}
This is the difference between two projections on to two different lines,hence by the same reasoning as in the proof of Lemma \ref{lem:diff between derivatives of Pi}  we obtain
\begin{equation*}
\abs{\frac{\partial \Pi_i^k}{\partial z_l}(x,y,z) - \frac{\partial \Pi_i^k}{\partial z_l}(0,y,z)} \leq c\abs{\nu_i(x) - \nu_i(0)}
\end{equation*}
as required.
\end{proof}
\section{Notation}

\begin{itemize}
\item $\mbbE[n]=\int_\Omega \frac{K}{2}|\nabla n|^2\,dx+K_{13}\int_{\partial\Omega}((n\cdot \nabla)n)\cdot\nu\,d\sigma  $
\item $\mcA:=W^{1,2}(\Omega;\bS^{d-1})\cap W^{2,1}(\Omega;\bS^{d-1})$
\item $\Omega \subset \RR^n$ a is $C^2$  domain (See Theorem \ref{thm:bdrydet})
\item $\nu$ unit-norm exterior normal.
\item $\mathcal{T}:= \{\gamma \in H^{\frac{1}{2}}(\partial \Omega , \SSS^2): \gamma(x)\cdot \nu(x) = 0 \text{ for almost every } x \in \partial \Omega\} $
\item $ \mcU :=\{u \in W^{1,2}(\Omega, \SSS^2): {\rm Trace}(u) \in \mathcal{T} \} $
\item $\mathbb{G}[n]:= \int_\Omega K\sum_{\alpha,\beta=1}^d \frac{\partial n^\alpha}{\partial x^\beta}\frac{\partial n^\alpha}{\partial x^\beta} - K_{13}\int_{\partial \Omega} \sum_{\alpha,\beta=1}^d \frac{ \partial \nu^\beta}{\partial x^\alpha} n^\beta n^\alpha.$
\item $\liptwo{\varphi_{r,a}}:=\max_{\abs{\alpha} = 2} \sup_{x \in \RR^2}\abs{\frac{\partial^{|\alpha|} \varphi_{r,a}(x)}{\partial x^{\alpha}}}$
\item $C(x,r):= \{y \in \RR^3:\abs{(y^1,y^2)-(x^1,x^2)}<r , \abs{y^3 -x^3} <r\}$
\item $\Omega_{\varphi} = \{(x^1,x^2,x^3)\in C(0,1) : x^3 < \varphi(x^1,x^2)\}$
\item $\Omega_0 := \{(x^1,x^2,x^3)\in C(0,1) : x^3 < 0\}$
\item $G_{\varphi} = \partial \Omega_{\varphi} \setminus \partial C(0,1)$
\item $G_{0} = \partial \Omega_{0} \setminus \partial C(0,1)$
\item $H_{\varphi} = \partial \Omega_{\varphi} \cap \partial C(0,1)$
\item $H_{0} = \partial \Omega_{0} \cap \partial C(0,1)$
\item $\mathbb{G}_{\varphi}[u] : = \int_{\Omega_{\varphi}} \abs{\nabla u}^2 \, dx - K_{13} \int_{G_\varphi} u^\alpha u^\beta \frac{\partial \nu^\beta }{\partial x^\alpha}\, dx.$
\item $B^d(0,1)$ is the ball in $\RR^d$ centred at $0$ with radius 1.

\end{itemize}

\end{appendices}

\section*{Acknowledgment.} The activity of Arghir Zarnescu on this work was partially
supported by a grant of the Romanian National Authority for Scientific Research and Innovation,
CNCS-UEFISCDI, project number PN-II-RU-TE-2014-4-0657

\bibliographystyle{acm}
\bibliography{BIB_K13}

\end{document}